\documentclass{amsart}
\usepackage{amsfonts,amssymb,mathtools,mathabx,mathrsfs}

\usepackage[normalem]{ulem}				

\usepackage[usenames,dvipsnames]{xcolor}
\usepackage[T1]{fontenc}	

\usepackage{tikz}
\usetikzlibrary{arrows, positioning, calc, intersections}
\usetikzlibrary{decorations.pathreplacing, decorations.markings}

\usepackage{hyperref}
\hypersetup{
    unicode=false,          		
    pdftoolbar=true,        		
    pdfmenubar=true,        	
    pdffitwindow=false,     		
    pdfstartview={FitH},    		
    linktocpage=true,			
    pdfnewwindow=true,      	
    colorlinks=true,       			
    linkcolor=blue,          		
    citecolor=PineGreen,    	
    filecolor=magenta,      		
    urlcolor=cyan           		
}

\theoremstyle{plain}
\newtheorem{theorem}{Theorem}[section]
\newtheorem{corollary}[theorem]{Corollary}
\newtheorem{lemma}[theorem]{Lemma}
\newtheorem{proposition}[theorem]{Proposition}

\theoremstyle{definition}
\newtheorem{definition}[theorem]{Definition}

\newtheorem{RHP}[theorem]{Riemann-Hilbert Problem}

\theoremstyle{remark}
\newtheorem{remark}[theorem]{Remark}

\numberwithin{figure}{section}
\numberwithin{equation}{section}

\allowdisplaybreaks

\usepackage{float}

\usepackage[font=small, format=plain, labelfont=bf, up, textfont=it, up]{caption}
\usepackage{subcaption}

\DeclareMathOperator{\ad}{ad}

\DeclareMathOperator{\imag}{Im}

\DeclareMathOperator{\sech}{sech}

\DeclareFontFamily{U}{mathx}{\hyphenchar\font45}
\DeclareFontShape{U}{mathx}{m}{n}{
      <5> <6> <7> <8> <9> <10>
      <10.95> <12> <14.4> <17.28> <20.74> <24.88>
      mathx10
      }{}
\DeclareSymbolFont{mathx}{U}{mathx}{m}{n}
\DeclareFontSubstitution{U}{mathx}{m}{n}
\DeclareMathAccent{\widecheck}{0}{mathx}{"71}
\DeclareMathAccent{\wideparen}{0}{mathx}{"75}

\begin{document}


\title[Global Existence for DNLS]{Global Existence for the Derivative Nonlinear Schr\"{o}dinger Equation with Arbitrary Spectral Singularities}
\author{Robert Jenkins}
\author{Jiaqi Liu}
\author{Peter Perry}
\author{Catherine Sulem}
\address[Jenkins]{Department of Mathematics, Colorado State University, Fort Collins, Colorado 80523-1874}
\address[Liu]{Department of Mathematics, University of Toronto, Toronto, Ontario M5S 2E4, Canada}
\address[Perry]{ Department of Mathematics,  University of Kentucky, Lexington, Kentucky 40506--0027}
\address[Sulem]{Department of Mathematics, University of Toronto, Toronto, Ontario M5S 2E4, Canada }
\date{\today}
\begin{abstract}

We show that the derivative nonlinear Schr\"odinger (DNLS)  equation is globally well-posed in the weighted Sobolev space $H^{2,2}(\mathbb{R})$. 
Our result exploits the complete integrability of DNLS and removes certain spectral conditions on the initial data required by our previous work \cite{JLPS17a},
thanks to Zhou's analysis on spectral singularities in the context of inverse scattering   \cite{Zhou89-2}. 
\end{abstract}
\maketitle
\tableofcontents

%
%
%
%

\newcommand{\purpletext}[1]{{\color{purple} #1}}
\newcommand{\bluetext}[1]{ #1 }								

%
%

\tikzset{->-/.style={decoration={
  markings,
  mark=at position .55 with {\arrow[scale=0.6]{triangle 45}} },postaction={decorate}}
}

%
%

\newcommand{\sidenote}[1]{\marginpar{\scriptsize{\color{purple} #1}}}

\newcommand{\eps}{\varepsilon}
\newcommand{\lam}{\lambda}

\newcommand{\darr}{\downarrow}
\newcommand{\rarr}{\rightarrow}

\newcommand{\Lam}{\Lambda}

\newcommand{\dee}{\partial}
\newcommand{\dbar}{\overline{\partial}}
\newcommand{\qbar}{\overline{q}}

\newcommand{\dotarg}{\, \cdot \, }

\newcommand{\ba}{\breve{a}}
\newcommand{\balpha}{\breve{\alpha}}
\newcommand{\bb}{\breve{b}}
\newcommand{\bbeta}{\breve{\beta}}
\newcommand{\br}{\breve{r}}
\newcommand{\bbC}{\mathbb{C}}

\newcommand{\zetabar}{\overline{\zeta}}

\newcommand{\scrB}{\mathscr{B}}

\newcommand{\C}{\mathbb{C}}
\newcommand{\N}{\mathbb{N}}
\newcommand{\R}{\mathbb{R}}
\newcommand{\Z}{\mathbb{Z}}

\newcommand{\calA}{\mathcal{A}}
\newcommand{\calB}{\mathcal{B}}
\newcommand{\calC}{\mathcal{C}}
\newcommand{\calG}{\mathcal{G}}
\newcommand{\calI}{\mathcal{I}}
\newcommand{\calL}{\mathcal{L}}
\newcommand{\calJ}{\mathcal{J}}
\newcommand{\calM}{\mathcal{M}}
\newcommand{\calS}{\mathcal{S}}
\newcommand{\calT}{\mathcal{T}}
\newcommand{\calQ}{\mathcal{Q}}

\newcommand{\One}{\mathbf{1}}

\newcommand{\kbar}{\overline{k}}
\newcommand{\ubar}{\overline{u}}
\newcommand{\zbar}{\overline{z}}

\newcommand{\bfA}{\mathbf{A}}
\newcommand{\bbR}{\mathbb{R}}
\newcommand{\bfr}{\mathbf{r}}
\newcommand{\bfM}{\mathbf{M}}
\newcommand{\bfN}{\mathbf{N}}
\newcommand{\bft}{\mathbf{t}}

\newcommand{\tildJ}{\widetilde{J}}
\newcommand{\brevcalJ}{\breve{\calJ}}

\newcommand{\dint}{\displaystyle{\int}}

\newcommand{\norm}[2][ ]{\left\Vert #2 \right\Vert_{#1}}
\newcommand{\bigO}[2][ ]{\mathcal{O}_{#1}\left( #2 \right)}

\newcommand{\upmat}[1]{\ttwomat{0}{#1}{0}{0}}
\newcommand{\lowmat}[1]{\ttwomat{0}{0}{#1}{0}}

\newcommand{\twomat}[4]
{
\begin{pmatrix}
#1	&	#2	\\
#3	&	#4
\end{pmatrix}
}


\newcommand{\ttwomat}[4]
{
\begin{pmatrix}
#1	&	#2	\\[3pt]
#3	&	#4
\end{pmatrix}
}

\newcommand{\Twomat}[4]
{
	\left(
			\begin{array}{cc}
			#1 & #2 \\
			\\
			#3 & #4
			\end{array}
	\right)
}

\newcommand{\diagmat}[2]
{
\begin{pmatrix}
#1		&	0	\\
0		&	#2
\end{pmatrix}
}

\newcommand{\offmat}[2]
{
\begin{pmatrix}
0	&	#1		\\
#2	&	0
\end{pmatrix}
}

\newcommand{\twovec}[2]
{
	\left(
		\begin{array}{c}
			#1		\\
			#2
		\end{array}
	\right)
}

\newcommand{\Twovec}[2]
{
	\left(
		\begin{array}{c}
			#1		\\
			\\
			#2
		\end{array}
	\right)
}

%
%

\section{Introduction}

In this paper, we prove global well-posedness of the Cauchy problem for the Derivative Nonlinear Schr\"odinger equation (DNLS)
\begin{equation}
\label{DNLS1}
\left\{
\begin{aligned}
i u_t + u_{xx} -i \eps (|u|^2 u)_x &=0, ~~\eps=\pm1\\
u(x,t=0) &= u_0(x)
\end{aligned}
\right.
\end{equation}
with initial condition $u_0$ in the weighted Sobolev space 
 $$ H^{2,2}(\R) = \left\{ u \in L^2(\R): u''(x), \, x^2 u(x) \in L^2(\R) \right\}.$$
In contrast to previous work using PDE methods \cite{FHI17,HO92,Wu15}, we impose no upper bound on the $L^2$-norm of the initial data (although we require more smoothness and decay than these authors), and in contrast to previous work using completely integrable methods \cite{JLPS17a,JLPS17b,Lee89,LPS15,PSS17}, we make no spectral restrictions to ``generic initial data''  that rule out singularities of the spectral data associated to the initial condition.
We use the complete integrability of DNLS discovered by Kaup and Newell \cite{KN78}. 
 As we will explain, an essential ingredient of our work is  Zhou's approach to inverse scattering with arbitrary spectral singularities \cite{Zhou89-2,Zhou98}, 
building on the work of Beals and Coifman \cite{BC84}; to our knowledge, the present paper constitutes the first application of these techniques to global wellposedness questions for integrable PDE's that involves \emph{no} spectral assumptions on the initial data. This is significant in that Zhou's methods are quite general and are likely applicable to wellposedness questions for other integrable PDE's in one space dimension.

To describe our results more precisely, we recall that the invertible gauge transformation
$$ \calG(u)(x) = u(x)  \bluetext{  ~\exp\left(  i \eps  \int_x^\infty |u(y)|^2 \, dy\right)}$$
maps solutions of  \eqref{DNLS1} to solutions of
\begin{equation}
\left\{
\begin{aligned}
\label{DNLS2}
iq_t + q_{xx} + i\eps q^2 \bar q_x + \frac{1}{2} |q|^4 q &= 0, ~~\eps=\pm1\\
q(x,t=0) &= q_0(x).
\end{aligned}
\right.
\end{equation}
Equation \eqref{DNLS2} is more directly amenable to inverse scattering. 
It is shown in \cite{CKSTT02}
 that $\calG$ is a continuous map from $H^{s}$ to $H^s$, $s>1/2$. It is  straightforward to check that it is a locally Lipschitz continuous map from $H^{2,2}(\R)$ to itself. 
Indeed,  denoting  $\varphi (u)= \int_x^\infty |u(y)|^2 \, dy$  and writing $\calG (u) = u+ \left( e^{i\eps\varphi}-1 \right) u$, it is sufficient to prove  the map 
$u \mapsto e^{i\eps \varphi}$ is Lipschitz continuous from $H^{2,2}(\R)$ into $W^{2,\infty}(\R)$  and apply the Leibnitz rule. In particular, one 
easily prove that 
$$ \|1 -e^{i \eps \int_x^\infty (|u(y)|^2-|v(y)|^2 )} dy \| _{W^{2,\infty} }\l  \lesssim \norm[H^{2,2}]{u-v} $$ where the implied constants may depend on $\norm[H^{2,2}]{u}$ and $\norm[H^{2,2}]{v}$.
Global wellposedness in $H^{2,2}(\R)$ for equations   \eqref{DNLS1} and \eqref{DNLS2} are thus equivalent.
In the following, we fix  
$\eps=-1$,  since solutions of \eqref{DNLS2} with 
$\eps=1$ are mapped to solutions of \eqref{DNLS2} with 
$\eps = -1$ by  $q(x,t)  \mapsto q(-x,t)$.
The main result of the paper is the following theorem:
\begin{theorem}
\label{Theorem-main}
Suppose that $q_0 \in H^{2,2}(\R)$. There exists a unique solution $q(x,t)$ of \eqref{DNLS2} with $q(x,t=0)=q_0$ and
$t \mapsto q(\dotarg,t) \in C([-T,T],H^{2,2}(\R))$ for every $T>0$. Moreover, the map $q_0 \mapsto q$ is Lipschitz continuous from $H^{2,2}(\R)$ to $C([-T,T], H^{2,2}(\R))$ for every $T>0$.
\end{theorem}

The Cauchy problem for equation \eqref{DNLS1}  is locally well-posed in $H^{1}(\R)$ as well as in weighted spaces $H^{m,0}\cap H^{0,m}$ ($m\ge1$) and it
is  globally well-posed for small  initial data  \cite{ TF80,HO92}. 
More precisely,  
it was proved  in  \cite{HO92} that for any initial condition  $u_0\in H^{1}(\R)$ such that $\| u_0 \|_{L^2} < \sqrt{2\pi}$, global wellposedness holds in $H^1(\R)$. 
The  smallness condition was recently improved  to  $\| u_0 \|_{L^2} < \sqrt{4\pi}$ (or $\|u_0\|_{L^2} = \sqrt{4\pi}$ with additional conditions on initial data)
\cite{Wu15, FHI17}.


The present paper also builds on previous work of the co-authors 
which proved global wellposedness of DNLS for initial conditions $u_0$ in
weighted Sobolev spaces  under some additional  conditions
that exclude
the so-called  \emph{spectral singularities}  \cite{JLPS17a,Liu17,PSS17}.
In this context,  we proved global  well-posedness  for data in an open and dense set of $H^{2,2}(\R)$  which allows finitely many resonances, which refer to  eigenvalues away from the continuous spectrum but no spectral singularities, and also
established the long-time behavior of solutions in the form of  the soliton resolution \cite{JLPS17b}. 
We will discuss precisely in Section \ref{sec:direct}  the meaning of spectral singularities. In the
present paper, we remove all
spectral assumptions on the initial data and obtain global well-posedness  of the DNLS equation for general initial condition in $H^{2,2}(\R)$. 

Our approach is inspired by the work of Zhou, 
who, in a series of papers \cite{Zhou89-2,Zhou95,Zhou98}, developed new tools to construct  direct and inverse scattering maps that are insensitive to singularities of the spectral data.
We emphasize that spectral singularities may affect the long-time behavior of solutions, in the same way that  eigenvalues affect 
the long-time behavior of solutions  through soliton resolution 
{(see \cite{JLPS17b} where the soliton resolution conjecture is proved for generic initial data). In the case of the  focusing cubic 
nonlinear  Schr\"odinger equation,
Kamvissis  \cite{Kamvissis96}  studied the effect of  a single spectral singularity on the large-time behavior of solutions. He showed that 
the latter is limited to 
the region of the  $(x,t)$-plane in which the spectral singularity is close to the point of stationary phase, and there,  slightly modifies the rate of decay.
In a future paper, we will investigate how spectral singularities affect the long-time behavior of DNLS solutions.
A new version of the inverse scattering transform has been recently introduced by Bilman and Miller \cite{BM19} to
deal with arbitrary-order poles and spectral singularities in the context of focusing NLS with non-zero boundary conditions. This method
relies on the initial value problem for the Lax pair  and avoids the use of a cut-off potential. }

 Occurrence of spectral singularities in the spectral problem is not an exceptional phenomenon.  In the context of the focussing NLS equation, Zhou   \cite{Zhou89-2} constructed one example in which Schwartz class potential leads to infinitely many eigenvalues accumulating on the real line to form a spectral singularity and another example where infinitely many spectral singularities accumulate. In Appendix B of \cite{JLPS17a}, we analyzed a  family of potentials of the form $q(x) = A \sech(x) e^{i\phi(x)}$  
 for which one can  explicitly compute the scattering  data, thus illustrating various characterizations of the spectral map.
In particular, we exhibit potentials  for which the associated spectral problem has either no discrete spectrum, or  exactly $n$ eigenvalues and no spectral singularities, or  $n$ eigenvalues and one spectral singularity. 

To explain our methods, we will sketch the completely integrable method for \eqref{DNLS2} as discovered by Kaup and Newell \cite{KN78} in two steps. First, we describe how the method works when the initial data do not support solitons or spectral singularities. Next, we describe how Zhou's method 
\cite{Zhou89-2,Zhou98} can be extended to the DNLS equation to construct global solutions in the presence of solitons and spectral singularities.

\subsection{The Inverse Scattering Method:  No Singularities}

Kaup and Newell \cite{KN78} showed that the flow determined by \eqref{DNLS2} may be linearized by spectral data associated to the linear problem
\begin{equation}
\label{LS}
\frac{d}{dx} \Psi(x,\zeta) = -i\zeta^2 \sigma \Psi + \zeta Q(x) \Psi + P(x) \Psi, \quad \zeta\in \bbR\cup i\bbR
\end{equation}
where $\Psi(x,\zeta)$ is a $2 \times 2$ matrix-valued function of $x$ and 
\begin{equation}
\label{sigma-Q-P}
\begin{gathered}
\sigma=\diagmat{1}{-1}, \\
Q(x) = \offmat{q(x)}{-\overline{q(x)}}, \quad
P(x) =  \bluetext{ \diagmat{p_1}{p_2} } =    \frac{i}{2} \diagmat{|q(x)|^2}{-|q(x)|^2}.
\end{gathered}
\end{equation}
Later, it will be convenient to set $\Psi(x,\zeta) = m(x,\zeta) e^{-ix \zeta^2 \sigma}$, so that $m$ solves 
the equation
\begin{equation}
\label{LS.m}
\frac{d}{dx} m (x,\zeta) = -i\zeta^2 \ad (\sigma) m + \zeta Q(x) m + P(x) m
\end{equation}
where
$$
\ad(\sigma) A  = \sigma A - A \sigma. 
$$
Equation \eqref{LS} admits bounded solutions provided $q \in L^1(\R) \cap L^2(\R)$ and $\zeta \in \R \cup i \R$. There exist unique solutions $\Psi^\pm(x,\zeta)$ of \eqref{LS} satisfying the respective boundary conditions
$$ \lim_{x \rarr \pm \infty} \Psi^\pm(x,\zeta)  e^{i\zeta^2 x \sigma} = I, \quad I= \diagmat{1}{1}. $$
These \emph{Jost solutions} have determinant $1$ and define action-angle variables $a$ and $b$ for the flow \eqref{DNLS2} through the relation
$$
\Psi^+(x,\zeta) = \Psi^-(x,\zeta) \twomat{a(\zeta)}{b(\zeta)}{\bb(\zeta)}{\ba(\zeta)}.
$$
That is, if $q(x,t)$ solves \eqref{DNLS2}, and $a(\zeta,t)$ and $b(\zeta,t)$ are the corresponding scattering data for $q(\cdot,t)$, then
\begin{equation}
\label{evolve-a-b}
\dot{a}(\zeta,t) = 0, \quad \dot{b}(\zeta,t)= -4i\zeta^4 b(\zeta,t).
\end{equation}
Thus, if the map $q \mapsto (a,b)$ can be inverted, one can hope to solve \eqref{DNLS2} via a composition of the direct scattering map $q \rarr (a,b)$, the flow map defined by \eqref{evolve-a-b}, and the inverse map $(a,b) \mapsto q$. 

The functions $a$ and $\ba$ have analytic extensions to the respective regions $\Omega^- = \{ \imag z^2 < 0 \}$ and $\Omega^+ = \{ \imag z^2 > 0\}$ (see Figure \ref{fig-1}). Zeros of $a$ (resp.\ $\ba$) in $\Omega^-$ (resp.\ $\Omega^+$) are associated to soliton solutions of \eqref{DNLS2}, while zeros of $a$ or $\ba$ on $\R \cup i \R$ are called \emph{spectral singularities}. For the moment, we assume that $a$ and $\ba$ are zero-free in their respective regions of definition. This allows us to define the reflection coefficients
\begin{equation}
\label{rba}
r(\zeta) = \bb(\zeta)/a(\zeta), \quad \br(\zeta) = b(\zeta)/\ba(\zeta), \quad \zeta\in \bbR\cup i\bbR. 
\end{equation}
The map $q \mapsto r$ is the \emph{direct scattering map}. One can recover $a$ and $b$ from $r$ by solving a scalar Riemann-Hilbert problem. By symmetry one  has that  that $\br(\zeta)=-\overline{r(\overline{\zeta})}$.

In his thesis, J.-H. Lee \cite{Lee83} formulated the inverse scattering map as a Riemann-Hilbert problem (RHP) in which $r$ and $\br$ enter as jump data for a piecewise analytic function. To describe it,  denote by $\R \cup i\R$  the oriented contour, shown in Figure \ref{fig-Sigma}, that bounds $\Omega^\pm$ with $\Omega^+$ to the left and $\Omega^-$ to the right. An oriented contour that divides $\C$ into two such regions $\Omega^+$ and $\Omega^-$, is called a \emph{complete} contour. 

Denote by $m^\pm$ the \emph{renormalized Jost solutions} $m^\pm = \Psi^\pm e^{ix\zeta^2 \sigma}$.  Let $m^+_{1}$ and $m^+_{2}$ denote the first and second columns of $m^+$, with a similar notation $m^-_{1}, m^-_{2}$ for the columns of $m^-$. From the integral equations \eqref{m1+}--\eqref{m2-},  it is easy to see that, for each $x$,  $m_{1}^-(x,\zeta)$ and $m_{2}^+(x,\zeta)$ extend to analytic functions of $z \in \Omega^+$, while
$m_{1}^+(x,\zeta)$ and $m_{2}^-(x,\zeta)$ extend to analytic functions of $z \in \Omega^-$. From these columns, 
one can construct left and right  \emph{Beals-Coifman solutions} $M(x,z)$ of \eqref{LS.m} which are piecewise analytic for $z \in \C \setminus (\R \cup i\R)$ and normalized so that $\lim_{x \rarr \infty} M(x,z) = I $ (right-normalized, \eqref{M-right}) or $\lim_{x \rarr -\infty} M(x,z) = I$ (left-normalized, \eqref{M-left}). In what follows, we discuss the right-normalized solution. Enforcing these normalizations involves division by $a$ and $\ba$ so any zeros of $a$ and $\ba$ would create new singularities.

The Beals-Coifman solution solves a Riemann-Hilbert problem (RHP) in the $z$ variable. Thus $x$ plays the role of a parameter and, for each $x$, the function $M(x,z)$ is piecewise analytic in $z$ with prescribed asymptotics as $z \rarr \infty$ and prescribed multiplicative jumps along the contour $\R \cup i\R$.






\begin{figure}[H]
\centering
\begin{subfigure}{0.45\textwidth}
\caption{The Regions $\Omega^\pm$ and the Contour $\R \cup i\R$
\label{fig-Sigma}
}

\medskip

\begin{tikzpicture}[scale=0.6]
\draw[thick,->-]	(0,0)	--	(-4,0);
\draw[thick,->-]	(0,0)	--	(4,0);
\draw[thick,->-]	(0,4)	--	(0,0);
\draw[thick,->-]	(0,-4)	--  (0,0);
\node at		(-2,2)		{$\Omega^-$};
\node at 	(2,-2)		{$\Omega^-$};
\node at  	(2,2)		{$\Omega^+$};
\node at		(-2,-2)	{$\Omega^+$};
\node [below] at (3.8,0)		{$\R$};
\node [right]  	at	(0,3.8)		{$i\R$};
\end{tikzpicture}
\end{subfigure}
\quad
\begin{subfigure}{0.45\textwidth}
\caption{Zeros of $a$ and $\ba$ 
\label{fig-Zeros}
}

\medskip

\begin{tikzpicture}[scale=0.6]
\draw[thick,->-]	(0,0)	--	(-4,0);
\draw[thick,->-]	(0,0)	--	(4,0);
\draw[thick,->-]	(0,4)	--	(0,0);
\draw[thick,->-]	(0,-4)	--  (0,0);
\draw[black,fill=black]	(1,1)	circle(0.1cm);
\draw[black,fill=black]	(1,-1)	circle(0.1cm);
\draw[black,fill=black]	(-1,1)	circle(0.1cm);
\draw[black,fill=black]	(-1,-1)	circle(0.1cm);
\node at		(0.5,0)		{$\times$};
\node at		(-0.5,0)		{$\times$};
\node at		(2,2)		{$\ba(\zeta)$};
\node at		(-2,-2)	{$\ba(\zeta)$};
\node at 	(2,-2)		{$a(\zeta)$};
\node at		(-2,2)		{$a(\zeta)$};
\node [below] at (3.8,0)		{$\R$};
\node [right]  	at	(0,3.8)		{$i\R$};
\end{tikzpicture}
\end{subfigure}
\caption{\label{fig-1} }
\end{figure}

More precisely, for each $x$, the piecewise analytic function $M(x,\dotarg)$ 
solves the following Riemann-Hilbert problem.
\begin{RHP}
\label{RHP1}
For each $x \in \R$, find an analytic\footnote{If $a$ has zeros, $M$ is meromorphic and discrete data for each pole must be added to close the problem. For the present, we assume that $a$ and $\ba$ are zero-free.} function $M(x,\cdot): \C \setminus (\R \cup i\R) \rarr SL(2,\C)$
with:
\begin{enumerate}
\item[(i)]		$\lim_{z \rarr \infty} M(x,z) = I $,
\smallskip
\item[(ii)]	$M$ has continuous boundary values $M_\pm$ as  $z\to \zeta\in \bbR\cup i\bbR$ from $\Omega^\pm$, and
\item[(iii)]	$M_\pm$ obey the jump relation 
				$$M_+(x,\zeta) = M_-(x,\zeta) e^{-ix\zeta^2 \ad \sigma} v(\zeta)$$
				where 
				$$e^{-ix\zeta^2\ad \sigma} v(\zeta) = \twomat{1+|r(\zeta)|^2}{e^{-2ix\zeta^2 }r(\zeta)}{- e^{2ix\zeta^2} \br(\zeta)}{1}.  $$
\end{enumerate}
\end{RHP}
The matrix   $ e^{-ix\zeta^2\ad \sigma} v$  is called the \emph{jump matrix} for the RHP \ref{RHP1}. We recover $q(x)$ through the asymptotic formula
\begin{equation}
\label{q-recon}
q(x) = 2i \lim_{z \rarr \infty} zM_{12}(x,z)
\end{equation}
 which may easily be deduced from the large-$z$-expansion for $M(x,z)$ and the fact that $M(x,z)$ satisfies \eqref{LS.m}.

RHP \ref{RHP1} and the reconstruction formula \ref{q-recon} define the \emph{inverse scattering map}.

\subsection{The Inverse Scattering Method:  Singularities}

So far, we have assumed that $a$ and $\ba$ are zero-free; however, zeros of $a$ and $\ba$ do occur for data of physical interest. By the symmetries
\begin{equation} \label{a.sym}
\ba(\zeta) = \overline{a(\zetabar)}, \quad 
a(-\zeta) = a(\zeta), \quad
\end{equation}
zeros of $a$ and $\ba$ in $\C \setminus (\R \cup i\R)$ occur in ``quartets'' as shown in Figure \ref{fig-Zeros}.
These quartets correspond to soliton solutions of \eqref{DNLS2}. 
The further symmetry
\begin{equation}
\label{b.sym}
b(-\zeta) = -b(\zeta), \quad \bb(\zeta) = -\overline{b(\zetabar)}
\end{equation}
and the determinant condition
$$
a(\zeta) \ba(\zeta) - b(\zeta) \bb(\zeta) = 1 
$$
imply that
$$
|a(it)|^2 - |b(it)|^2 = 1
$$
for all real $t$, so $a$ has no zeros on the imaginary axis. However, zeros of $a$ on the real axis may occur and correspond to \emph{spectral singularities}. RHP \ref{RHP1} is no longer solvable since  the jump matrix $v$ now has singularities on the contour $\R \cup i\R$; moreover, any zeros of $a$ and $\ba$ in their domains of analyticity will make the Beals-Coifman solutions meromorphic rather than analytic. 

On the other hand, any zeros of $a$ and $\ba$ lie in the disc $$B(0,R) = \{ z: |z| < R \},$$ where $R$ is determined by $\norm[H^{2,2}]{q}$ (see, for example, \cite[Proposition 3.2.5]{Liu17}). Moreover, for $\norm[H^{2,2}]{q}$ sufficiently small, $a$ and $\ba$ are zero-free on their respective domains. We will say that such a potential has \emph{zero-free scattering data}.

Zhou's insight in \cite{Zhou89-2,Zhou98} is that RHP \ref{RHP1} can be modified in the following way.  First, choose $R$ so large that $a$ and $\ba$ have no zeros in $\C \setminus B(0,R)$, and denote by $\Sigma_R$ the circle of radius $R$ centered at $0$. 

Choose $x_0 >0$ sufficiently large so that the potential
\begin{equation}
\label{cutoff-q}
q_{x_0}(x) = 	\begin{cases}
						0,		&	x \leq x_0\\
						q(x)	&	x > x_0
					\end{cases}
\end{equation}
has zero-free scattering data; a sufficient condition to achieve this is that
$$ \sup_{|z| \leq R} \norm[L^1(x>x_0)]{z Q + P} < 1/2$$ 
(see Section \ref{sec:direct}, \eqref{x0-condition} and the discussion that follows).

 Note that both $x_0$ and $R$ may be chosen \bluetext{  uniformly}  for $q$ in a bounded subset of $H^{2,2}(\R)$. Next, let $M^{(0)}(x,z)$ denote the solution to RHP \ref{RHP1} for $q_{x_0}$.  The function $M^{(0)}$ is analytic in $\C \setminus (\R \cup i\R)$
with continuous boundary values $M^{(0)}_\pm$ on $\R \cup i\R$. 
Indeed,  resonances  and spectral singularities for $q_{x_0}$ are ruled out by the small norm assumption.

\begin{remark}
Although the sharp cutoff potential $q_{x_0}$ is not in the $H^{2,2}$ space, it is in $H^{0,2}$ and we will only need this decay property to construct $H^{2,0}$ scattering data on a bounded set.
\end{remark}

Denote by $M^{(1)}$ the unique solution of the Volterra integral equation
\begin{equation}
\label{M1.def}
 M^{(1)}(x,\zeta) = I + 
	\int_{x_0}^x e^{i(y-x)\zeta^2 \ad\sigma}  
	\left( 
	   \zeta Q(y)M^{(1)} (y,\zeta)+
	   P(y) M^{(1)}(y,\zeta) 
	\right) dy 
\end{equation}
and  define
\begin{equation}
\label{M2.def}
	M^{(2)}(x,\zeta) =M^{(1)}(x,\zeta)e^{-i(x-x_0)\zeta^2\ad\sigma } M^{(0)}(x_0,\zeta).
\end{equation}

Since $M^{(2)}$ and $M^{(0)}$ agree at $x=x_0$, it follows by uniqueness that $M^{(2)}(x,\zeta) = M^{(0)}(x,\zeta)$ for all $x \geq x_0$. We notice that $M^{(1)}(x_0, z)$ is entire in $z$, thus $M^{(2)}(x_0, z)$ and $M^{(0)}(x_0, z)$ share the same domain of analyticity.
Define the contour
\[
	\Sigma = \R \cup i \R \cup \Sigma_\infty, \qquad \Sigma_\infty = \{ |z| = R \}
\]
oriented as in Figure~\ref{aug-1}, and define
\begin{equation}
\label{m-aug}
   \bfM(x,z) = \begin{cases}
     M(x,z),		&   z \in \C \setminus ( B(0,R) \cup \Sigma) \\
     M^{(2)}(x,z),	&   z \in B(0,R) \setminus \Sigma,
   \end{cases}
\end{equation}
The function $M(x,z)$ is piecewise analytic on $\C  \setminus \left( B(0,R) \cup (\R \cup i\R)\right)$ because $a$ and $\ba$ are zero-free for $|z| > R$. 
By construction, the function $M^{(2)}$ is piecewise analytic in $B(0,R) \setminus (\R \cup i\R)$. The new unknown $\bfM(x,z)$
obeys RHP \ref{RHP.zeta}. 

The jump matrix of  the Riemann-Hilbert problem for $\bfM(x,z)$  is unchanged outside the circle $\Sigma_\infty$ but is replaced inside by new jump data that may be explicitly computed from $q_0$ and $q_{x_0}$; see Section \ref{sec:direct} for a full discussion. Since $\bfM(x,\zeta) = M(x,\zeta)$ in a neighborhood of infinity, we can still recover $q$ from the reconstruction formula \eqref{q-recon}. To carry out the analysis, we change variables from $\zeta$ to $\lam=\zeta^2$ and actually analyze RHP  \ref{RHP.lambda}.

\begin{figure}[H]
\begin{subfigure}[t]{0.40\textwidth}
\captionsetup{singlelinecheck=off}
\caption[.]{The augmented contour $$\Sigma = \R \cup i\R \cup \Sigma_\infty$$ in the $\zeta$-plane}

\bigskip

\begin{tikzpicture}[scale=0.42]
\fill[fill=gray!20!] (6,0) -- (4,0)  arc (0:90:4) -- (0,6) -- (6,6) -- (6,0) -- cycle;
\fill[fill=gray!20!] (-6,0) -- (-4,0)  arc (-180:-90:4) -- (0,-6) -- (-6,-6) -- (-6,0) -- cycle;
\fill[fill=gray!20!] (0,4)  arc (90:180:4) -- (0,0) -- (0,4);
\fill[fill=gray!20!] (0,-4)  arc (-90:0:4) -- (0,0) -- (0,-4);
\draw[thick,->-]	(0,0)		--	(0,4);
\draw[thick,->-]	(-4,0)		--	(0,0);
\draw[thick,->-]	(4,0)		--	(0,0);
\draw[thick,->-]	(0,0)		--	(0,-4);
\draw[thick,->-]	(0,-6)		--	(0,-4);
\draw[thick,->-]	(-4,0)		--	(-6,0);
\draw[thick,->-]	(4,0)		--	(6,0);
\draw[thick,->-]	(0,6)		--	(0,4);
\draw[thick,->-]	(0,4)				arc(90:180:4);
\draw[thick,->-]	(0,4)				arc(90:0:4);
\draw[thick,->-]	(0,-4)				arc(270:180:4);
\draw[thick,->-]	(0,-4)				arc(270:360:4);
\node at		(45:4.5)					{\scriptsize{$+$}};
\node at		(45:3.5)					{\scriptsize{$-$}};
\node at		(135:3.5)				{\scriptsize{$+$}};
\node at		(135:4.5)				{\scriptsize{$-$}};
\node at		(225:3.5)				{\scriptsize{$-$}};
\node at 	(225:4.5)				{\scriptsize{$+$}};
\node at		(315:3.5)				{\scriptsize{$+$}};
\node at		(315:4.5)				{\scriptsize{$-$}};
\node[right]	at	(0.2,2)			{\scriptsize{$-$}};
\node[left]		at	(-0.2,2)			{\scriptsize{$+$}};
\node[above]	at	(-2,0.2)			{\scriptsize{$+$}};
\node[below]	at	(-2,-0.2)			{\scriptsize{$-$}};
\node[above]	at	(2,0.2)			{\scriptsize{$-$}};
\node[below] 	at	(2,-0.2)			{\scriptsize{$+$}};
\node[left]		at	(-0.2,-2)			{\scriptsize{$-$}};
\node[right]	at	(0.2,-2)			{\scriptsize{$+$}};
\node[left]		at	(-0.2,5)			{\scriptsize{$-$}};
\node[right]	at	(0.2,5)			{\scriptsize{$+$}};
\node[above]	at	(5,0.2)			{\scriptsize{$+$}};
\node[below]	at	(5,-0.2)			{\scriptsize{$-$}};
\node[left]		at	(-0.2,-5)			{\scriptsize{$+$}};
\node[right]	at	(0.2,-5)			{\scriptsize{$-$}};
\node[above]	at	(-5,0.2)			{\scriptsize{$-$}};
\node[below]	at	(-5,-0.2)			{\scriptsize{$+$}};
\node[above] at (4,4)				{\footnotesize{${\widetilde{\Omega}}_1$}};
\node[above] at (-4,4) 				{\footnotesize{${\widetilde{\Omega}}_2$}};
\node[above] at (-4,-4)			{\footnotesize{${\widetilde{\Omega}}_3$}};
\node[above] at (4, -4)				{\footnotesize{${\widetilde{\Omega}}_4$}};
\node[above] at (30:2)				{\footnotesize{${\widetilde{\Omega}}_5$}};
\node[above] at (150:2) 			{\footnotesize{${\widetilde{\Omega}}_6$}};
\node[below] at (210:2)			{\footnotesize{${\widetilde{\Omega}}_7$}};
\node[below] at (330:2)			{\footnotesize{${\widetilde{\Omega}}_8$}};
\end{tikzpicture}
\label{aug-1}
\end{subfigure}
\qquad
\begin{subfigure}[t]{0.40\textwidth}
\captionsetup{singlelinecheck=off}
\caption[.]{The  augmented contour $$\Gamma=\R \cup \Gamma_\infty$$ in the $\lambda$-plane}

\bigskip

\begin{tikzpicture}[scale=0.42]
\draw[thick,white]	(0,-6) -- (0,6);
\fill[fill=gray!20] (-6,0) -- (-3,0) arc (180:0:3) -- (6,0) -- (6,6) -- (-6,6) -- (-6,0) -- cycle;
\fill[fill=gray!20] (-3,0) arc (-180:0:3) -- (-3,0) -- cycle; 
\draw[black, fill=black] 		(3,0) 		circle [radius=0.15];
\draw [black, fill=black] 		(-3,0) 	circle [radius=0.15];
\draw[thick,->-]		(-3,0) arc(180:0:3);
\draw[thick,->-]		(-3,0)	 arc(180:360:3);
\draw [thick,->-] 		(-6,0) -- (-3,0);
\draw [thick,->-]		(3,0)	--	(-3,0);
\draw [thick,->-] 		(3,0) -- 	(6,0);
\node [below] at (1,-0.2) 		{\scriptsize{$+$}};
\node [above] at (1,0.2) 		{\scriptsize{$-$}};
\node [below] at (5,-0.2) 		{\scriptsize{$-$}};
\node [above] at (5,0.2) 		{\scriptsize{$+$}};
\node [below] at (-5,-0.2) 	{\scriptsize{$-$}};
\node [above] at (-5,0.2) 		{\scriptsize{$+$}};
\node [below] at (0,-0.5) 		{$\Omega_4$};
\node [above] at (0,0.5) 		{$\Omega_3$};
\node [below] at (4,-3) 		{$\Omega_2$};
\node [above] at (4, 3) 		{$\Omega_1$};
\node [above] at (140:3.5)	{$\Gamma_\infty^+$};
\node [below] at (220:3.5) 	{$\Gamma_\infty^-$};
\node[below] at (2.2,0)		{\scriptsize{$S_\infty$}};
\node[below] at (-2.0,0)		{\scriptsize{$-S_{\infty}$}};
\end{tikzpicture}
\label{aug-2}
\end{subfigure}
\caption{
The Augmented Contour $\Gamma$ for the Modified Riemann-Hilbert Problem RHP~\ref{RHP.lambda} and its preimage $\Sigma$ in the $\zeta$-plane. The regions $\Omega_+ = \Omega_1 \cup \Omega_4$ (shaded) and $\Omega_- = \Omega_2 \cup \Omega_3$ lie, respectively, to the left and right of $\Gamma$.
}
\end{figure}
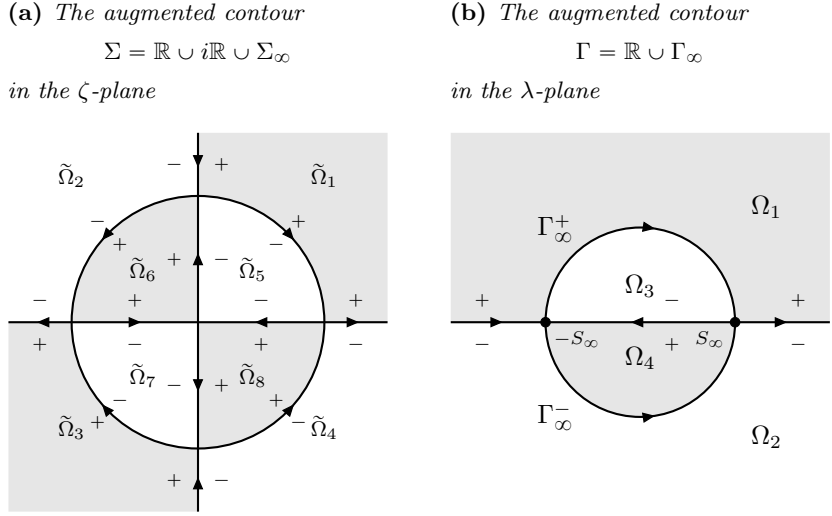

To analyze the direct map (from the given potential $q_0$ to the jump matrix for the augmented contour $\Sigma$) and the inverse map (from the jump matrix to the recovered potential) it is helpful to exploit the symmetry reduction of the spectral problem \eqref{LS} to the spectral variable $\lam = \zeta^2$. Under the map $\zeta \mapsto \zeta^2$, 
 the augmented contour $\Sigma$ is mapped to the contour
\begin{equation}\label{gamma.def}
	\Gamma = \R \cup \Gamma_\infty, \qquad \Gamma_\infty = \{ |z| = S_\infty\}
\end{equation}
with induced orientation as shown in Figure~\ref{aug-2}; the shaded and unshaded regions shown in Figure~\ref{aug-1} are mapped to the shaded and unshaded regions shown in Figure~\ref{aug-2}. The circle $\Sigma_\infty$ is mapped to $\Gamma_\infty$, the circle of radius $S_\infty = R^2$; we let $\Gamma_\infty^\pm = \Gamma_\infty \cap \C^\pm$. The augmented contour $\Gamma$ in Figure~\ref{aug-2} decomposes $\C \setminus \Gamma$ into two sets\footnote{The notation for the sets $\Omega_\pm$, consistent with our use of subscripts for boundary values, should not be confused with the superscripted sets $\Omega^\pm = \{ \pm \Im z^2 > 0 \}$ previously introduced.  
\begin{equation}
	\Omega_+ = \Omega_1 \cup \Omega_4 
	\quad \text{and} \quad
	\Omega_- = \Omega_2 \cup \Omega_3 	
\end{equation}
such that $\Omega_+$ (resp. $\Omega_-$) lies everywhere to the left (resp. right) of $\Gamma$. 
The contour ${\Gamma}$ can be viewed simultaneously as the boundary of $\Omega_+$ or $\Omega_-$, and we will write 
\begin{equation}
	\Gamma_+ = \partial \Omega_+ 
	\quad \text{or} \quad
	{\Gamma}_- = \partial \Omega_- 
\end{equation}
when we want to emphasize either interpretation. 
} Finally, in what follows, we will set
\begin{equation}
\label{R-infty}
\R_\infty = \R \setminus [-S_\infty,S_\infty],
\end{equation}
the part of the contour $\R$ outside the \bluetext{circle $\Gamma_\infty$.}

In the rest of the paper, the letter $z$ is used as  a general notation for  a complex variable off contours, while  $\zeta$ refers the variable on the contour $\Sigma$ and $\lambda=\zeta^2$ to the variable on the contour $\Gamma$.

One can compute the jump data for the Riemann-Hilbert problem on the contour $\Gamma$ explicitly in terms of scattering data for $q$, scattering data for $q_{x_0}$, and normalized Jost solutions for $q$ (see Figure \ref{fig-scattering} and Proposition \ref{matrices-entry}; it is then easy to show that the direct spectral map from $q \in H^{2,2}(\R)$ to these scattering data is continuous in a natural topology on the jump data (see Theorem \ref{thm:scattering data} for a precise statement). 

It remains to show that the scattering data can be time-evolved continuously and that RHP with scattering data as described in Theorem \ref{thm:scattering data} can be uniquely solved and used to recover the potential $q$. To do so, much as in \cite{JLPS17a} and \cite{LPS15}, we show that the Riemann-Hilbert problem in the $\lam$ variables is equivalent to a Riemann-Hilbert problem in the $\zeta$-variable which is uniquely solvable. We then apply Zhou's uniqueness theorem (see Proposition \ref{schwarz invariance} and \cite{Zhou89-1}) to obtain unique solvability. We also need to show that the recovered potential is continuous in the scattering data; this will  follow from  Zhou's results  \cite{Zhou98} and our previous results on the scattering transform in \cite{LPS15}.


Finally, we sketch the content of the paper. 

Section  \ref{sec:direct} is devoted to the direct scattering map. 
In Section \ref{sec:scattering-prob-zeta}, we recall the basic properties of the scattering problem  and Beals-Coifman solutions in the $\zeta$ variables. In Sections \ref{sec:scattering-zeta} and \ref{sec:scattering-lambda}, we construct the scattering data in the $\zeta$ and $\lam$ variables. 
The goal is to choose the scattering data so the inverse  scattering problem will allow a reconstruction formula for the potential.  
For this purpose, we implement Zhou's method to deal with spectral singularities. 
In this setting, the usual Beals-Coifman solutions are changed to piecewise analytic functions  according to \eqref{m-aug}. We give explicit formulas for the  corresponding jump matrices along the augmented contours.\footnote{In \cite{Zhou98}, the non-zero off-diagonal entries are  not calculated explicitly.} We use Zhou's approach \cite{Zhou98} (see also  Trogdon-Olver \cite{TO16}) to address the matching conditions at  the intersection points of the contours and give a full description of the jump matrices and their factorization.
In Section  \ref{time-evol}, we establish the time evolution of the scattering data.  Finally, as shown in \cite[Lemma 3.4]{DZ03}  in the absence of spectral singularities, right and left RHPs are needed  to obtain decay rate of the potential as $x\to \pm\infty$, there are separate left and right augmented RHPs for the same purpose in this paper. 
In section \ref{aux-scatt}, we compute the auxiliary matrix that relates  their corresponding  jump matrices (see \eqref{auxiliary}). This result allows us to focus on the right RHP thereafter.

In Section \ref{sec:unique-solv}, we  show that the RHP with the augmented contour and  the  jump matrices, as derived in Sections   \ref{sec:scattering-zeta} and \ref{sec:scattering-lambda} has a unique solution. The proof follows the lines of the proof given in \cite{Liu17}.
Suppose the RHP in $\lam$ has a null vector $\bfN$, i.e., a solution which satisfies the jump conditions but vanishes as $z \to \infty$. 
This null vector corresponds to a homogeneous solution $\nu$ of the Beals-Coifman integral equation for the RHP, and induces a homogeneous solution $\mu$ to the Beals-Coifman integral equation for the RHP in the $\zeta$ variable 
which is the zero solution due to Zhou's vanishing lemma \cite[Theorem 9.3]{Zhou89-1}.
It follows from Fredholm theory that the Beals-Coifman equation for $\mu$ is uniquely solvable, and hence that the RHP in the $\lam$ variable is also uniquely solvable. 

As in \cite{JLPS17a}, we establish the existence and uniqueness of solutions to the RHP for scattering data
in a larger space  $Y$ (see Definition \ref{def-space-Y}) in order to obtain uniform resolvent estimates  for scattering data in bounded sets of a smaller space.

In Section \ref{sec:inverse}, we   establish the mapping properties of the inverse scattering map and estimate the potential obtained from   the reconstruction formula in the  $\lam$ variable.
This analysis requires another technical step taken from Zhou's method \cite{Zhou98}. As seen  in Fig. \ref{aug-2}, the orientation of the piece of the contour $(S^-_\infty, S^+_\infty)$
goes from right to left. A second augmentation  shown in Figure \ref{fig:new.mod.contour} allows the new contour to have the usual orientation thus allowing standard estimates of the Cauchy projectors on $\mathbb{R}$ to be used to obtain decay estimates on the potential. The Lipschitz continuity follows from the second resolvent identity.

To analyze Riemann-Hilbert problems with self-intersecting contours, we make use of certain Sobolev spaces of functions that obey continuity conditions at self-intersection points. For the reader's convenience, we briefly describe these Sobolev spaces in Appendix \ref{app:Sobolev}.
In Appendix \ref{app-cc}, we present the necessary abstract functional analysis tools used to prove uniform resolvent estimates needed for the Lipschitz continuity
of the inverse scattering map presented in Section \ref{sec:inverse}.

We end the introduction by discussing the role that  factorization of the jump 
matrix plays in our application of the  Beals-Coifman approach to inverse scattering. In Figure  
\ref{aug-2}, the oriented contour divides the complex plane into positive and 
negative regions. We factorize the jump matrix 
$$J(\lambda)=J_-(\lambda)^{-1} J_+(\lambda)$$
where 
$W_+ = J_+ -I$ and $W_- = I- J_-  $ belong to $H^1(\Gamma_\pm)$ and are continuous across the intersections between straight line contours and 
circular arcs, respecting the orientations.
This continuity means that the  matrix pair $(W_+, W_-)$ belongs to a pair of decomposing algebras $(H^1(\Gamma_+), H^1(\Gamma_-))$ in the sense of  Zhou; see \cite[Sec.\ 9]{Zhou89-1}) where a general theory of Riemann-Hilbert problems on self-intersecting contours is presented. 
As shown by Zhou, this decomposition implies that the Beals-Coifman integral operator \eqref{SIE-nu} is Fredholm. Unique solvability of \eqref{SIE-nu} then follows from the Fredholm alternative and  an appropriate vanishing lemma (the statement that the homogeneous version of \eqref{SIE-nu} has no nonzero solutions).

In our case,
we need to show that the Beals-Coifman operator
$$C_J f= C_\Gamma^+ (f W_-) +C_\Gamma^- (f W_+)$$
is Fredholm. For this purpose, following Zhou, we  approximate 
$W_\pm$ by rational functions; in this approximation, the 
operator $C^\pm_{W_\mp} \circ C^{\mp}_{W_\pm}$ is compact. We thus obtain a Fredholm regulator of the  Beals-Coifman operator (see \cite[Prop. 4.1]{Zhou89-1}).  Another way to think about the compactness is that 
continuity across intersection points prevents singularities near these points 
which might otherwise occur, spoiling the compactness. 

%
%

\section{The Direct Scattering Map}
\label{sec:direct}

\subsection{\texorpdfstring{ The scattering problem in the $\zeta$ variable}{The scattering problem in the zeta variable}}
\label{sec:scattering-prob-zeta}

The system \eqref{LS.m} can be  written  in the form of an integral equation for the $2\times 2$ matrix $m(x,\zeta)$
\begin{equation}
\label{IE-m}
m(x,\zeta)=I+\int_{\delta}^x e^{i(y-x)\zeta^2 \ad\sigma}  \left( \zeta Q(y)m(y,\zeta)+P(y)m(y,\zeta) \right)dy,
\end{equation}
where the lower limit $\delta$ can be different for  various choices of normalization. We will use several solutions of \eqref{IE-m}. The standard AKNS method starts with the following two Volterra integral equations as special cases of \eqref{IE-m} for  $\imag \zeta^2=0$:
$$
m^{\pm}(x, \zeta)=I + \int_{\pm\infty}^x e^{i(y-x)\zeta^2 \ad\sigma}  \left( \zeta Q(y)m^{\pm} (y,\zeta)+P(y)m^{\pm}(y,\zeta) \right) dy,
$$
which are expressed in componentwise form as
\begin{align}
\label{m1+}
\twovec{m_{11}^+(x,\zeta)}{m_{21}^+(x,\zeta)}
	&=	\twovec{1}{0}
			-\int_x^\infty 
				\twovec{\zeta qm _{21}^+ + p_1 m_{11}^+}
							{e^{2i\zeta^2 (x-y)}
								\left[ 
									-\zeta\overline{q}m_{11}^+ + p_2m_{21}^+ 
								\right]}				
				\, dy	
\\[10pt]
\label{m2+}
\twovec{m_{12}^+(x,\zeta)}{m_{22}^+(x,\zeta)}
	&=	\twovec{0}{1}
			-\int_x^\infty
				\twovec{e^{-2i\zeta^2(x-y)}
								\left[\zeta q m_{22}^++ p_1 m_{12}^+ \right]}
							{-\zeta \qbar m_{12}^+ + p_2 m_{22}^+}
					\, dy					
\\[10pt]
\label{m1-}
\twovec{m_{11}^-(x,\zeta)}{m_{21}^-(x,\zeta)}
	&=	\twovec{1}{0}
			+\int_{-\infty}^x
				\twovec{\zeta q m _{21}^- + p_1 m_{11}^-}
							{e^{2i\zeta^2 (x-y)}
								\left[ 
									-\zeta\overline{q}m_{11}^- + p_2m_{21}^- 
								\right]}				
				\, dy	
\\[10pt]
\twovec{m_{12}^-(x,\zeta)}{m_{22}^-(x,\zeta)}
	&=	\twovec{0}{1}
			+ \int_{-\infty}^x
				\twovec{e^{-2i\zeta^2(x-y)}
					\left[\zeta q m_{22}^-+ p_1 m_{12}^- \right]
					}
					{-\zeta\qbar m_{12}^- + p_2 m_{22}^-}
				\, dy.
\label{m2-}
\end{align}
By  uniqueness theory for ODEs and the normalizations of $m^\pm$ as $x\to\pm\infty$,  
 $m^\pm(x,\zeta)$ defined by \eqref{m1+}--\eqref{m2-}
are related by a matrix $A(\zeta)$ with $\det A(\zeta)=1$  in the form
$$
m^{+}(x,\zeta)=m^{-}(x,\zeta) e^{-ix\zeta^2\ad\sigma}A(\zeta), \qquad A(\zeta)=\twomat{a}{\bb}{b}{\ba}.
$$
The matrix-valued function $A(\zeta)$  is expressed in terms of  $m^{(\pm)}$ as
\begin{align}
\label{a.int}
a(\zeta)	
	&=	1-\int_{-\infty}^\infty \left(\zeta q m_{21}^+ +  p_1 m_{11}^+\right)\, dy=1+\int_{-\infty}^\infty \left(-\zeta \qbar m_{12}^- + p_2 m_{22}^-\right) \, dy,\\
\label{ba.int}
\ba(\zeta)
	&=	1-\int_{-\infty}^\infty \left(-\zeta \qbar m_{12}^+ + p_2 m_{22}^+ \right) \, dy=1+\int_{-\infty}^\infty \left(\zeta q m_{21}^- + p_1 m_{11}^-\right) \, dy,
\end{align}
and
\begin{align*}
b(\zeta)	
	&=	\int_{-\infty}^\infty  e^{-2i\zeta^2 y}\left( \zeta\qbar m_{11}^+ -  p_2 m_{21}^+\right)\, dy 
	=	\int_{-\infty}^\infty e^{-2i\zeta^2 y}\left( \zeta\qbar m_{11}^- -  p_2 m_{21}^-\right)\, dy,\\
\bb(\zeta)
	&=	-\int_{-\infty}^\infty e^{2i\zeta^2 y} \left(\zeta q m_{22}^+ + p_1 m_{12}^+ \right) \, dy 
	=-\int_{-\infty}^\infty  e^{2i\zeta^2 y}  \left(\zeta q m_{22}^- + p_1 m_{12}^-\right) \, dy.
\end{align*}
We  now construct the Beals-Coifman solutions needed  for the RHP in the form of piecewise analytic matrix functions. 
The left and right Beals-Coifman solutions are constructed
from the normalized Jost solutions as follows:
\begin{align}
\label{M-right}
M_R(x,z)	&=	
	\begin{cases}
	\left[ 	\dfrac{m_1^{-}(x,z) }{\ba(z)}, \, m_2^{+ }(x,z) \right],	
			&	\imag  z^2 > 0	\\
			\\
	\left[	m_1^{+}(x,z), \dfrac{m_2^{ - }(x,z)}{a(z)} \right],
			&	\imag   z^2 < 0
	\end{cases}	\\[5pt]
	\label{M-left}
M_L(x,z)	&=
	\begin{cases}
	\left[	m_1^{ -}(x,z), \dfrac{m_2^{+}(x,z)}{ \ba(z)} \right],
			&	\imag   z^2 > 0	\\
			\\
	\left[	\dfrac{m_1^{ +}(x,z)}{a(z)}, m_2^{ -}(x,z) \right],
			&	\imag   z^2 < 0.
	\end{cases}
\end{align}
The Beals-Coifman solutions are piecewise meromorphic with continuous boundary values denoted $M_{L,\pm}$ and $M_{R,\pm}$ as $\pm \imag z^2 \darr 0$, in the absence of spectral singularities.
The Beals-Coifman solutions corresponding to the potential $q_{x_0}$ are constructed similarly. 

From here onward, we will analyze the right Beals-Coifman solution \eqref{M-right}  and drop the subscripts $R$ and $L$.  The left RHP is connected to the right RHP through multiplication by an auxiliary scattering matrix which is constructed in Section \ref{aux-scatt}.

\subsection{\texorpdfstring{Construction of the scattering data in the $\zeta$ variable}{Construction of the scattering data in the zeta variables}}
\label{sec:scattering-zeta}

In this subsection, we construct  
 the piecewise analytic function $M^{(2)}(x,z)$ introduced in    \eqref{m-aug} and  defined inside the circle $\Sigma_\infty$ from which one extracts scattering data in the form of jump matrices along the contour $\Sigma_\infty$. In this subsection, $M$ denotes the right-normalized Beals-Coifman solution $M_R$.

Combining  \eqref{m1-} and  \eqref{m2+}, we obtain 
\begin{equation}
\label{m-a}
\left[ m_1^{-} , m_2^{+} \right]= I+ \int_{\delta}^{x} e^{i(y-x)\zeta^2 \ad\sigma}  \left( \left(   \zeta Q(y)+P(y)  \right) \left[ m_1^{-}(y) , m_2^{+}(y) \right]   \right) dy
\end{equation}
where  $\delta$ is chosen differently for the different entries of the matrix, namely $\delta=-\infty$ for the (1-1) and (2-1) entries and $\delta=+\infty$ for the (1-2) and (2-2) entries. Using \eqref{ba.int}, we  rewrite \eqref{m-a} as 
\begin{multline*}
\left[ m_1^{-} , m_2^{+} \right]= \\
 \twomat{\ba}{0}{0}{1} + \int_{\delta}^{x} e^{i(y-x)\zeta^2 \ad\sigma}  \left( \left(   \zeta Q(y)+P(y)  \right) \left[ m_1^{-}(y) , m_2^{+}(y) \right]   \right) dy
\end{multline*}
where $\delta=-\infty$ for the (2-1) entry and $\delta=+\infty$ for the (1-1), (1-2) and (2-2) entries. If the inverse of $ \twomat{\ba}{0}{0}{1}$ exists, 
we obtain a Fredholm equation for $M$ defined in  \eqref{M-right}:
\begin{equation}
\label{Fredholm}
\left[ 	\dfrac{m_1^{-} }{\ba}, \, m_2^{+ } \right]=I+  \int_{\delta}^{x} e^{i(y-x)\zeta^2 \ad\sigma}  \left( \left(   \zeta Q(y)+P(y)  \right) \left[ 	\dfrac{m_1^{-} }{\ba}, \, m_2^{+ } \right]  \right) dy
\end{equation}
and $\left[ {m_1^{-} }/{\ba}, \, m_2^{+} \right]$ solves  \eqref{Fredholm} iff $\ba(\zeta)\neq 0$.


The right-normalized Beals-Coifman solution $M$ is analytic in the intersection of $\pm \imag z^2 >0$ and $|z| > R$, where $R$ is chosen so large that any zeros of $a$ and $\ba$ are contained inside the disc $B(0,R) = \{z: |z| \leq R\}$. We now show how to construct solutions $M^{(2)}(x,\zeta)$, analytic inside this disc, and modify the Riemann-Hilbert problem accordingly. 

Recall from \eqref{cutoff-q}, for $x_0 \gg 1$,  denote by $q_{x_0}$ 
$$ q_{x_0} (x) = 
	\begin{cases}
		0,		&	x 	\le 	x_0	\\
		q(x)	&	x	> 		x_0
	\end{cases}
$$
and denote by $Q_{x_0}$ and $P_{x_0}$ the matrices \eqref{sigma-Q-P} with $q$ replaced by $q_{x_0}$. We choose $x_0$ so that
\begin{equation}
\label{x0-condition}
\sup_{\zeta \in B(0,R)}  \norm[L^1]{\zeta Q_{x_0} + P_{x_0}} < 1/2.
\end{equation}
This condition 
guarantees that there is 
a bounded Beals-Coifman solution $M^{(0)}$ normalized as $x\to \infty$ associated to the  potential $q_{x_0}$ in the form of \eqref{M-right}.
To see this, first note that the equation
\begin{equation}
\label{IE-m-x0}
m(x,\zeta)= I +\int_{\delta}^x e^{i(y-x)\zeta^2 \ad\sigma}  \left( \zeta Q_{x_0}(y)m(y,\zeta)+P_{x_0}\bluetext{ (y)}m(y,\zeta) \right)dy,
\end{equation}
(where $\delta=-\infty$ for the (2-1) entry and $\delta=+\infty$ for the (1-1), (1-2) and (2-2) entry) is uniquely solvable for $\zeta \in B(0,R)$ owing to the smallness condition \eqref{x0-condition}. 

Next, we claim that 
that  $\ba_0(z)$ associated to $q_{x_0}$ is nonzero for $z \in \Omega^+ \cap B(0,R)$; 
we prove this statement  by contradiction. Suppose that there exists $z_0$ such that \eqref{x0-condition} holds and $m(x,z_0)$ solves \eqref{IE-m-x0}, but $\ba(z_0)=0$. By uniqueness, \bluetext{ there exists a non-singular matrix $B(z_0)$ such that}
$$m(x, z_0)=\left[ m_1^{-} , m_2^{+} \right] e^{-i x z_0^2 \ad\sigma}B(z_0).$$
Here, $m_1^{-} , m_2^{+} $ are the Jost solutions \eqref{m2+} and \eqref{m1-} associated to the potential $q_{x_0}$.
Letting $x\to +\infty$ and using \eqref{ba.int}, we obtain
$$\twomat{1}{0}{*}{1}=\twomat{0}{0}{*}{1}e^{-i x z_0^2 \ad\sigma}B(z_0),$$
which leads to a contradiction.
Thus the cutoff potential $q_{x_0}$ supports neither eigenvalues nor spectral singularities in $B(0,R)$, so that we can construct a bounded Beals-Coifman solution of the form \eqref{M-right} associated to the potential $q_{x_0}$ and normalized as $x \to \infty$. We denote by $M^{(0)}$ this unique bounded solution.

Using the solutions $M(z)$ and $M^{(0)}$ corresponding to to the initial data potential $q$ and the related potential $q_{x_0}$, respectively, one defines a new functon $\bfM$ using Zhou's constructions as described in \eqref{M1.def}-\eqref{m-aug} above. The matrix $\bfM$ is analytic in $\C \setminus \Sigma$, and we can compute 
the jump matrix
$$v(\zeta)=e^{ix \zeta^2 \ad\sigma}\bfM_-(x,\zeta)^{-1}\bfM_+(x,\zeta)$$
explicitly across the various parts of the augmented contour $\Sigma$.
%
Along the contour $\bbR\cup i\bbR$, outside of the circle,  
\begin{equation}
\label{v}
v(\zeta)=\Twomat{1-r(\zeta)\br(\zeta)}{r(\zeta)}{-\br(\zeta)}{1} ~.
\end{equation}
Along the contour $\bbR\cup i\bbR$ inside of the circle, 
$$
v(\zeta)=\Twomat{1}{-r_0(\zeta)}{\br_0(\zeta)}{1-r_0(\zeta)\br_0(\zeta)} ~.
$$
Here, the subscript ``0'' denotes the scattering data generated by  $q_{x_0}$.

Since both $M$ and $M^{(2)}$ are solutions of \eqref{LS.m} with non-vanishing determinant,
we have
\begin{equation}
\label{v0}v(\zeta)=e^{ix\zeta^2\ad\sigma} \left( M^{(2)}(x,\zeta)^{-1}M(x,\zeta) \right)
\end{equation}
along the circle $\Sigma_\infty$.
In particular, setting $x=x_0$, we obtain $v(\zeta)$  in terms of Jost functions. 
Across the arc in the first and third quadrant, we have:
\begin{align}
\label{v-circle-+}
e^{-ix_0\zeta^2\ad\sigma} v(\zeta) &= M^{(2)}(x_0,\zeta)^{-1}M(x_0,\zeta) =\Twomat{1}{0}{\dfrac{m^{-}_{21} (x_0, \zeta) }{\ba(\zeta) \ba_0(\zeta)  } }{1}.
\end{align}
Across the arc in the second and fourth quadrant, we have:
\begin{align}
\label{v-circle--}
e^{-ix_0\zeta^2\ad\sigma} v(\zeta) 
				&= M^{(2)}(x_0,\zeta)^{-1}M(x_0,\zeta) 
				&=\Twomat{1}{-\dfrac{m^{-}_{12} (x_0, \zeta) }{a(\zeta) a_0(\zeta)  }}{0}{1}.
\end{align}
Denote by $A^\dagger$ the hermitian conjugate of the matrix $A$. The following property of $v$ will be used later to prove the unique solvability of the RHP (Proposition \ref{vanishing-zeta}).
 \begin{proposition}
 \label{schwarz invariance}
The jump matrix $v$  \bluetext{along the contour}  $\Sigma$, defined in \eqref{v}-\eqref{v-circle--}, satisfies:
\begin{enumerate}
\item[(i)] $v(\zeta)+v(\zeta)^\dagger$ is positive definite for $\zeta\in\bbR$.
\smallskip
\item[(ii)] $v(\overline{\zeta})=v(\zeta)^\dagger$ for $\zeta\in \Sigma \setminus\bbR$.
\end{enumerate}
\end{proposition}

\begin{proof}
\bluetext{ This is an  immediate consequence of the definitions  \eqref{v}--\eqref{v-circle--} and  \eqref{rba} as well as the symmetries \eqref{a.sym}--\eqref{b.sym}.   }
\end{proof}

\subsection{\texorpdfstring{Construction of the scattering data in the $\lambda$ variable}{Construction of the scattering data in the lambda variable}}
\label{sec:scattering-lambda}

In  the absence of eigenvalues and spectral singularities, we reduced the scattering problem \eqref{LS.m} of $\zeta \in \R \cup i\R$ to scattering problem for $\lambda = \zeta^2 \in \bbR$, and identified a single scattering datum $\rho(\lam)$ defining the direct scattering map \cite{LPS15}; we can carry out a similar reduction here.  
Let $m(x,\zeta)$ be a solution to \eqref{LS.m}.  We   set
$$ m^\sharp(x,\zeta) = \Twomat{m_{11}(x,\zeta)}{ \zeta^{-1} m_{12}(x,\zeta)}{\zeta m_{21}(x,\zeta) }{m_{22}(x,\zeta)}.$$
 $m^\sharp$ is an even function of $\zeta$.  Defining
$ \lambda = \zeta^2$, 
$n(x,\lambda) = m^\sharp(x,\zeta)$, the map
$$ \twomat{a}{b}{c}{d} \mapsto \twomat{a}{\zeta^{-1}b }{\zeta c }{d} $$
is an automorphism of $2 \times 2$ matrices and commutes with differentiation in $x$. 
It follows that the functions $n^\pm$ obtained from $m^\pm$ by this map obey
\begin{subequations}
\label{n}
\begin{align}
\label{n.de}
\frac{dn^\pm}{dx}	&=	-i\lambda \ad \sigma  (n^\pm) +\offmat{q}{-\lam \qbar} n^\pm + P n^\pm\\
\label{n.ac}
\lim_{x \rarr \pm \infty} n^\pm(x,\lam)	&=	I
\end{align}
\end{subequations}
and satisfy
\begin{align}
\label{n.T}
n^+(x,\lam)& =	n^-(x,\lam)
	e^{-i\lambda x \ad \sigma } 
	\Twomat{\alpha(\lam)}
				{\beta(\lam)}
				{\lam {\bbeta(\lam)}}
				{{\balpha(\lam)}}\\[5pt]
\nonumber
           & =  n^-(x,\lam)
	e^{-i\lambda x \ad \sigma } 
	\Twomat{\alpha(\lam)}
				{\beta(\lam)}
				{-\lam \overline{\beta(\lam)}}
				{\overline{\alpha(\lam)}}
\end{align}
where
$ \alpha(\lam) = a(\zeta), \quad \beta(\lam) = \zeta^{-1} \bb(\zeta) $
and the relation 
$  |\alpha(\lam)|^2 +\lam |\beta(\lam)|^2 = 1 $ holds. 

In the presence of  spectral singularities, we perform the change of variable $\zeta\to \lambda$ in the same way as in \cite{LPS15} and obtain the corresponding row vector-valued Beals-Coifman solutions $N^{(0)}$, $N^{(2)}$ and $N$ :
\begin{align*}
N^{(0)} &= \textrm{first row of } \twomat{\zeta^{-1/2}}{0}{0}{\zeta^{1/2}}M^{(0)}  \twomat{\zeta^{1/2}}{0}{0}{\zeta^{-1/2}},\\
N^{(2)} &= \textrm{first row of } \twomat{\zeta^{-1/2}}{0}{0}{\zeta^{1/2}}M^{(2)}  \twomat{\zeta^{1/2}}{0}{0}{\zeta^{-1/2}},\\
N &= \textrm{first row of } \twomat{\zeta^{-1/2}}{0}{0}{\zeta^{1/2}}M \twomat{\zeta^{1/2}}{0}{0}{\zeta^{-1/2}}.
\end{align*}
The contour $\Gamma$ for the new RHP, defined by \eqref{gamma.def} is the image in Figure~\ref{aug-2} of the contour $\Sigma$ in Figure \ref{aug-1} under the change of variable $\lam = \zeta^2$.

 Notice that the direction of  the contour that consists of the part of the real axis inside the circle is from right to left.
Define  the piecewise analytic function $\bfN$ as
\begin{equation}
\label{bfN}
\bfN(x,z)	=
	\begin{cases}
		N(x,z),		& z \in \Omega_1\cup \Omega_2,\\
		N^{(2)}(x,z), 	& z \in \Omega_3\cup \Omega_4.
	\end{cases}
\end{equation}
By setting
\begin{align*}
\alpha(\lambda) &=a(\zeta), \quad \balpha(\lambda) =\ba(\zeta),\\
\rho(\lambda)&= \zeta^{-1}r(\zeta), \quad \rho_0(\lambda)= \zeta^{-1}r_0(\zeta),\\
n_{21}^-(x,\lambda)&=\zeta  m_{21}^-(x,\zeta),\quad n_{12}^-(x,\lambda)=\zeta^{-1}  m_{12}^-(x,\zeta),
\end{align*}
we obtain from \eqref{v} -- \eqref{v-circle--} the  jump matrices $J(\lam)$ for the piecewise  row vector $\bfN$.
  
 \begin{proposition}\label{matrices-entry}
The jump matrices for $\bfN$ along the various parts of the contour ${\Gamma}$ are given as follows:
 $$\bfN_+(x,\lambda)=\bfN_-(x,\lambda)e^{-i\lambda \ad\sigma}J(\lambda)$$
 
\begin{enumerate}
\item[(i)] on $\R_\infty$ the part of the real line outside the circle:
\begin{equation}
\label{J}
J(\lambda)=\ttwomat{ 1+\lambda|\rho(\lambda)|^2}{\rho(\lambda)}{\lambda\overline{\rho(\lambda)}}{1}
\end{equation}
\item[(ii)] on $(-S_\infty, S_\infty)$ the part of the real line inside the circle:
\begin{equation}
\label{J-0}
J(\lambda)=\ttwomat{1} {-\rho_0(\lambda)}{-\lambda\overline{\rho_0(\lambda)}}  { 1+\lambda|\rho_0(\lambda)|^2}
\end{equation}
\item[(iii)] on the semicircular arc $\Gamma_\infty^+$ in $\bbC^+$:
\begin{equation}
\label{J+}
J(\lambda)=\ttwomat{1}{0}{e^{-2ix_0\lambda}  \dfrac{n^{-}_{21} (x_0, \lambda) }{\balpha(\lambda) \balpha_0(\lambda)  }}{1}
\end{equation}
\item[(iv)] on the semicircular arc $\Gamma_\infty^-$ in $\bbC^-$:
\begin{equation}
\label{J-}
J(\lambda)=\ttwomat{1}{-e^{2ix_0\lambda}\dfrac{n^{-}_{12} (x_0, \lambda) }{\alpha(\lambda) \alpha_0(\lambda)  }}{0}{1} ~.
\end{equation}
\end{enumerate}
\end{proposition}

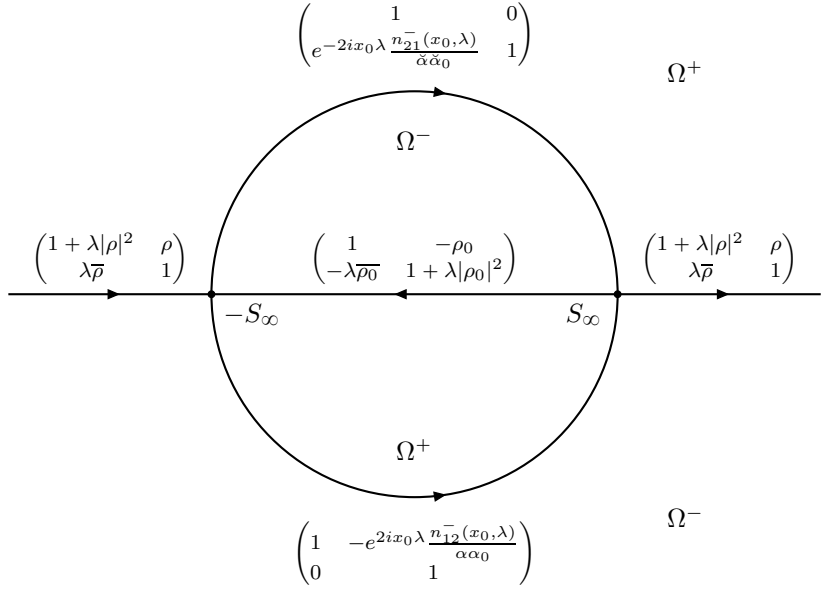
\begin{figure}[H]
\caption{Scattering data for $q$}

\bigskip

\begin{tikzpicture}[scale=0.9]
\draw[black, fill=black] 		(3,0) 		circle [radius=0.05];
\draw [black, fill=black] 		(-3,0) 	circle [radius=0.05];
\draw[thick,->-]		(-3,0) arc(180:0:3);
\draw[thick,->-]		(-3,0)	 arc(180:360:3);
\draw [thick,->-] 		(-6,0) -- (-3,0);
\draw [thick,->-]		(3,0)	--	(-3,0);
\draw [thick,->-] 		(3,0) -- 	(6,0);
\node [below] at (0,-2) 		{$\Omega^+$};
\node [above] at (0,2) 		{$\Omega^-$};
\node [below] at (4,-3) 		{$\Omega^-$};
\node [above] at (4, 3) 		{$\Omega^+$};
\node[below] at (2.5,0)		{$S_\infty$};
\node[below] at (-2.4,0)		{$-S_{\infty}$};
\node[above] at 	(-4.5,0)		{\footnotesize{$\twomat{1+\lam|\rho|^2}{\rho}{\lam \overline{\rho}}{1}$}};
\node[above] at 	(4.5,0)		{\footnotesize{$\twomat{1+\lam|\rho|^2}{\rho}{\lam \overline{\rho}}{1}$}};
\node[above]	at	(0,0)			{\footnotesize{$\twomat{1}{-\rho_0}{-\lam \overline{\rho_0}}{1+\lam|\rho_0|^2}$}};
\node[above]	at	(0,3.2)			{\footnotesize{$\twomat{1}{0}{e^{-2ix_0 \lam}\frac{n_{21}^-(x_0,\lam)}{\balpha \balpha_0}}{1}$}};
\node[below] 	at	(0,-3.2)			{\footnotesize{$\twomat{1}{-e^{2ix_0\lam}\frac{n_{12}^-(x_0,\lam)}{\alpha \alpha_0}}{0}{1}$}};
\end{tikzpicture}
\label{fig-scattering}
\end{figure}

\begin{remark}
The \textit{scattering data} associated to the potential $q$  are defined as the entries of the different  jump matrices
along ${\Gamma}$, as listed in Proposition \ref{matrices-entry} and shown in Figure \ref{fig-scattering}. We show the  factorizations of jump matrices exist and obtain estimates in appropriate Sobolev spaces. The choice of scattering data is motivated by the inverse problem. From these spectral data,   we will, in the next section,  define an inverse map and a reconstruction of the potential. Note that the scattering data depend on the choice of $x_0$ as well as the choice of the large circle $\Gamma_\infty$.  Indeed in \cite{Zhou89-2},  scattering data are seen  as an equivalence class. In the study of the inverse map, we will need the fact that the reconstruction formula does not depend on $x_0$ and $\Gamma_\infty$. This is because the reconstruction formula involves a limit as $\lambda$ tends to infinity of the entry $(1,2)$ of  the solution of a RHP, and  will not be affected by the exact position of the cut-off point or the circle $\Gamma_\infty$, although the RHP itself depends on it. For more details, we refer  to \cite[Theorem 3.3.15]{Zhou89-2}. 
 \end{remark}
 
To give a  full characterization of  the scattering data,  we use \bluetext{   the Sobolev } spaces  $H^k_z(\Gamma)$ and $H^k_\pm(\Gamma)$ 
\bluetext{ defined on self-intersecting contours}  (see Appendix \ref{app:Sobolev})
and   the notion of $k$-regularity \cite[Definition 2.54]{TO16} of a given jump matrix along an admissible contour.
\bluetext{ All contours under consideration here}  are admissible in the sense of \cite[Definition 2.40]{TO16}.

\begin{definition}
A jump matrix $J$ defined on an admissible contour $\Gamma$ is $k$-\emph{regular} if $\Gamma$ is  complete and $J$ has a factorization
$$J(s)=J_-^{-1}(s)J_+(s)$$
where $J_\pm(s)-I$ and $J_\pm^{-1}(s)-I$ $\in H^k_\pm(\Gamma)$.
\end{definition}
\begin{definition}
Assume $a\in\gamma_0$, the set of self-intersections of $\Gamma$. Let $\Gamma_1$, . . . , $\Gamma_m$ be a counter-clockwise ordering of sub-components of $\Gamma$ which contain $z = a$ as an endpoint. For $J \in H^k(\Gamma)$, we define $\widehat{J}_i$ \bluetext{ as the restriction }  $J\restriction_{\Gamma_i}$ if $\Gamma_i$ is oriented outwards and by $(J\restriction_{\Gamma_i})^{-1}$ otherwise. We say that  {\emph{$J$ satisfies the $(k -1)$th-order product condition}} if, using the $(k -1)$th-order Taylor expansion of each $J_i$, we have
\begin{equation}
\label{product}
\prod_{i=1}^m \widehat{J}_i=I+\mathcal{O}\left(| \lambda -a|^k  \right)\quad \forall a\in \gamma_0~.
\end{equation}
\end{definition}
The following theorem is due to Zhou \cite{Zhou99}; see also Trogdon-Olver  \cite[Theorem 2.56]{TO16}.

\begin{theorem}
\label{thm:k-regular}
The two following statements are equivalent:

(i) $J-I$ and $J^{-1}-I$ $\in H^k(\Gamma)$ away from points of self intersection and $J$ satisfies the $(k-1)$th-order product condition;

(ii) $J$ is $k$-regular.
\end{theorem}

In the next theorem, we check that the jump matrix $J(\lambda)$
satisfies the condition (i) of Theorem \ref{thm:k-regular} and characterize the large-$\lam$ decay of scattering data in  weighted Sobolev spaces.  
Let
\begin{align*}
H^{2,2}(\dee\Omega_2)	 	&=	
		\left\{ f \in H^2(\dee \Omega_2):  
					\left. f \right|_{\R_\infty} \in H^{2,2}(\R_\infty)
		\right\},	\\
H^{1,1}(\dee \Omega_1)		&=	
		\left\{	f \in H^1(\dee \Omega_1):
					\left. f \right|_{\R_\infty} 		\in H^{1,1}(\R_\infty)
		\right\}.
\end{align*}
Theorem \ref{thm:scattering data} should be compared to (C2.28) of \cite{Zhou95}, where the scattering matrix is characterized as belonging to $H^{k}$ for any $k\geq 1$ given initial data $q_0$ is in Schwartz class. Theorem \ref{thm:scattering data} shows that the direct scattering transform maps  a potential $q$ in the weighted Sobolev space $H^{2,2}(\R)$ into scattering data in appropriate \emph{weighted} Sobolev spaces.

\begin{theorem}
\label{thm:scattering data}
The matrix $J(\lambda)$ admits a triangular factorization $$J(\lambda)=J_-^{-1}(\lambda)J_+(\lambda)$$ where: 
\begin{enumerate}
\item[(i)]
$J_{-}(\lambda)-I\in H^{2,2}(\partial \Omega_2) $, $J_{-}(\lambda)-I\in H^{2}(\partial \Omega_3) $,
$J_{+}(\lambda)-I\in H^{2}(\partial \Omega_4)$ and $J_+(\lambda)-I\in H^{1,1}(\partial\Omega_1)$\footnote{ The asymmetry of the regularity properties of  the  terms  in  $J(\lambda)$ factorization  on the various parts of the contour comes from the fact that the (1,2) and (2,1) entries in the expression for $J$ (see eq.  \eqref{J}) differ by a weight  $\lambda$.}, and
\smallskip
\item[(ii)] $J_+\restriction_{\partial\Omega_1}-I$ and $J_-\restriction_{\partial\Omega_3}-I$ are strictly lower triangular while $J_-\restriction_{\partial\Omega_2}-I$ and $J_+\restriction_{\partial\Omega_4}-I$ are strictly upper triangular.
\end{enumerate}
\begin{enumerate}
\item[(iii)]  The matrix $J(\lam)$ satisfies the first-order product condition at the 
\bluetext{ intersection} points $\pm S_\infty$ \bluetext{ of} the real $\lam$-axis.
\end{enumerate}
\end{theorem}
\begin{proof}
 Let $J_i$  be the restriction of $J$ to  $\Gamma_i$, $1 \leq i \leq 5$, 
where the contours $\Gamma_i$ are shown in the figure below
\begin{center}
\begin{figure}[H]
\caption{Zero-Sum Conditions}
\bigskip
\begin{subfigure}{0.4\textwidth}
\caption{$-S_\infty$}
\bigskip

\begin{tikzpicture}[scale=0.6]
\draw[thick,->-,>=stealth]		(-6,0) -- (-3,0);
\draw[thick,->-,>=stealth]		(0,0)	--	(-3,0);
\draw[thick,->-,>=stealth]		(-3,0)		arc(180:90:3);
\draw[thick,->-,>=stealth]		(-3,0)		arc(-180:-90:3);
\draw[black,fill=black]			(-3,0)		circle(0.10cm);
\node[below]	at	(-3.75,-0.15){$-S_\infty$};
\node[above]	at	(-5,0)			{$\Gamma_1$};
\node[above]	at	(135:3.5)	{ $\Gamma_2$};
\node[below]	at	(225:3.5)	{ \bluetext{  $\Gamma_4$}};
\node[above]	at	(-1.5,0)		{\bluetext{ $\Gamma_3$}};
\end{tikzpicture}
\end{subfigure}
\qquad
\begin{subfigure}{0.4\textwidth}
\caption{$S_\infty$}
\bigskip

\begin{tikzpicture}[scale=0.6]
\draw[thick,->-,>=stealth]		(3,0)	--	(0,0);
\draw[thick,->-,>=stealth]		(3,0)	--	(6,0);
\draw[thick,->-,>=stealth]		(0,3)		arc(90:0:3);
\draw[thick,->-,>=stealth]		(0,-3)		arc(-90:0:3);
\node[below]	at	(3.5,-0.15){$S_\infty$};
\draw[black,fill=black]			(3,0)		circle(0.10cm);
\node[above]	at	(5,0)			{\bluetext{ $\Gamma_1$}};
\node[above] at	(45:3.5)		{$\Gamma_2$};
\node[below]	at	(-45:3.5)	{  \bluetext{ $\Gamma_4$}};
\node[above]	at	(1.5,0)		{ \bluetext{ $\Gamma_3$}};
\end{tikzpicture}
\end{subfigure}
\label{fig-zero-sum}
\end{figure}
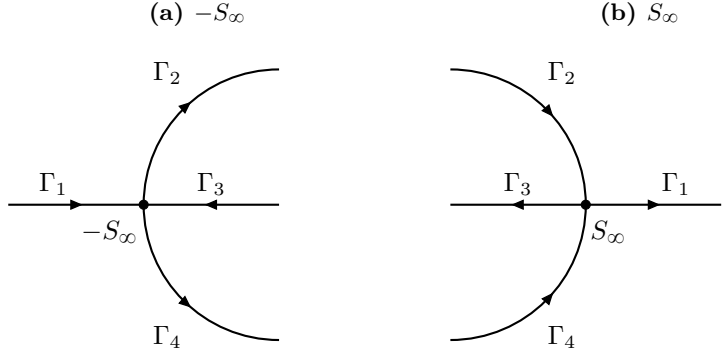
\end{center}
and denote  $\R_\infty = \R \setminus [-S_\infty,S_\infty]$.  
On $\R_\infty$, the scattering matrix $J_1$ admits the  factorization:
\begin{equation}
\label{factor-J1}
J_1(\lambda) 
	=J_{1,-}(\lambda)^{-1}J_{1,+}(\lambda)
	=\twomat{1}{\rho(\lambda)}{0}{1}\twomat{1}{0}{\lambda\overline{\rho(\lambda)}}{1},
\end{equation}
with the same decomposition for $J_5(\lam)$, 
while on $(-S_\infty, S_\infty)$, the scattering matrix $J_3$ admits the  factorization:
\begin{equation}
\label{factor-J3}
J_3(\lambda)
	=J_{3,-}(\lambda)^{-1}J_{3,+}(\lambda)
	=\twomat{1}{0}{-\lambda\overline{\rho_0(\lambda)}}{1}\twomat{1}{-\rho_0(\lambda)}{0}{1}.
\end{equation}

\medskip

(i) The methods of section 3 of \cite{LPS15}  can be used to show that $\rho \in H^{2,2}(\R_\infty)$. It follows from this fact and the explicit factorization \eqref{factor-J1} that  
$$\left. \left[J_{1,-}(\lam)-I\right] \right|_{\R_\infty} \in H^{2,2}(\R_\infty) \;\; \textrm{and} \;\;
\left. \left[J_{1,+}(\lam)-I\right]\right|_{\R_\infty} \in H^{1,1}(\R_\infty).$$
\bluetext{Similarly,} the restrictions of $J_\pm(\lam)-I$ to $(-S_\infty,S_\infty)$ belong to $H^2(-S_\infty,S_\infty)$. The remaining Sobolev estimates all involve the bounded semicircular contours $\Gamma_\infty^\pm$. 
$\Gamma_\infty^\pm$ are open, finite length contours, so $H^2(\Gamma_\infty^\pm)$ is equivalent to $H^2(0,\pm\pi) $ after parametrization by an angle $\theta$.  The $H^2$-norm of a function $f$  controls the $L^\infty$-norm of $f$ and $f'$ and  thus $H^2$ is an algebra by the Leibnitz rule.
Using \eqref{J+} and \eqref{J-}, it suffices to show that  the functions 
$n_{12}^-(x_0,\lam)$, $1/\balpha(\lam)$, and $1/\balpha_0(\lam)$ belong  to $H^2(\Gamma^+_\infty)$
and that the functions $n_{21}^-(x,\lam)$,  $1/\alpha(\lam)$, and $1/\alpha_0(\lam)$ belong to $H^2(\Gamma^-_\infty)$.
This is easily deduced from the Volterra integral equations corresponding to \eqref{n} and the integral representations  for $\alpha$ and $\balpha$ that can be deduced from \eqref{n.T}.

\medskip

(ii)  
The assertions about triangularity follow from the factorizations \eqref{factor-J1} and \eqref{factor-J3} together with the formulas
\eqref{J+} and \eqref{J-}. 

\medskip

(iii)  
Using the relation  \eqref{n.T},  the scattering matrices $J_2$ and $J_4$ are given at the self-intersection point $S_\infty$ respectively by
\begin{align*}
J_2(S_\infty) 
	&=	\twomat{1}{0}{e^{-2ix_0 S_\infty}  \dfrac{n^{-}_{21} (x_0,  S_\infty) }
						{\balpha( S_\infty) \balpha_0( S_\infty)  }}{1}\\[5pt]
     &=	\twomat{1}{0}{e^{-2ix_0 S_\infty} \dfrac{ n^{+}_{21}(S_\infty) }
     					{\balpha_0(S_\infty)  } }{1}  
     		\twomat{1}{0}
     					{  \dfrac{S_\infty\overline{\beta (S_\infty)} n^+_{22} (S_\infty)}
     						{\balpha (S_\infty)\balpha_0(S_\infty) } }
     					{1} \\[5pt]
     &=	\twomat{1}{0}{-S_\infty\overline{\rho_0(S_\infty)} }{1}
     		\twomat{1}{0}{S_\infty\overline{\rho(S_\infty)}}{1}
\end{align*}
and
\begin{align*}
J_4(S_\infty) 
	&=	\twomat{1}{-e^{2ix_0 S_\infty}  \dfrac{n^{-}_{12} (x_0,  S_\infty) }{\alpha( S_\infty) \alpha_0( S_\infty)  }}
						{0}{1}\\[5pt]
	&=	\twomat{1}{-e^{2ix_0 S_\infty} \dfrac{ n^{+}_{12}(S_\infty) }{\alpha_0(S_\infty) }}
						{0} {1} 
			\twomat{1}
						{
								\dfrac{ \beta (S_\infty)
										n^+_{11} (S_\infty) }{ \alpha (S_\infty)\alpha_0(S_\infty) }  
						}
						{0}{1}\\[5pt]
	&=	\twomat{1}{\rho(S_\infty)}{0}{1}\twomat{1}{-\rho_0(S_\infty)}{0 }{1}~.
\end{align*}
The factorizations of $J_2$ and $J_4$ along the arcs are obtained by polynomial  interpolation between $S_{\pm \infty}$ (see for example 
eq.(5.19) of \cite{Liu19}).

We want to establish \eqref{product} for $k=2$, that is,
\begin{equation}
\label{prod1}\prod_{i=1}^4 \widehat{J}_i=I+\mathcal{O}\left(|\lambda-S_\infty|^2  \right).
\end{equation}
Denoting  by $\mathfrak{J}_i$ the first order Taylor polynomial of $J_i$ at $S_\infty$ ,  $i=1,...,4$, proving \eqref{prod1} is equivalent to
proving that
$$\mathfrak{J}_1\mathfrak{J}_2^{-1}=\mathfrak{J}_4\mathfrak{J}_3^{-1}+\mathcal{O}\left(|\lambda-S_\infty|^2  \right).$$
It is clear that $J_1(S_\infty)J_2(S_\infty)^{-1}=J_4(S_\infty)J_3(S_\infty)^{-1}$. We  also have to check that 
$$\left(J_1J_2^{-1} \right)_\lambda(S_\infty)=\left(J_4J_3^{-1} \right)_\lambda(S_\infty).$$
To achieve this, we need to show that 
\begin{equation}
\label{factor'-1}
\dfrac{d}{d\lambda} \dfrac{ e^{-2ix_0 \lambda}   n^{-}_{21} (x_0,  \lambda )}{\balpha( \lambda) \balpha_0( \lambda)  }\bigg\rvert_{\lambda=S_\infty} ={\rho'(S_\infty)}{-\rho'_0(S_\infty)}
\end{equation}
This is done by first letting $\lambda\to \bbR$, which leads to the factorization of $J_4$:
\begin{align*}
	\twomat{1}{-e^{2ix_0 \lambda}  \dfrac{n^{-}_{12} (x_0,  \lambda) }{\alpha( \lambda) \alpha_0( \lambda)  }}{0}{1}=	\twomat{1}{\rho(\lambda)}{0}{1}\twomat{1}{-\rho_0(\lambda)}{0 }{1}~.
\end{align*}
and take derivative along $\bbR$. In the same way, we can show
$$ \dfrac{d}{d\lambda} \dfrac{ e^{2ix_0 \lambda}   n^{-}_{12} (x_0,  \lambda )}{\alpha( \lambda) \alpha_0( \lambda)  }\bigg\rvert_{\lambda=S_\infty}={-S_\infty\overline{\rho'_0(S_\infty)} }-\overline{\rho_0(S_\infty)} +
     		{S_\infty\overline{\rho'(S_\infty)}}+\overline{\rho(S_\infty)}.$$ 		
We thus verify  the $(k-1)$th order product condition for $k=2$ and Part $(i)$  of Theorem \ref{thm:k-regular}holds for the matrix $J$.  We conclude that $J$ is $k$-regular, which in turn implies that $J_+(\lambda)$ satisfies the matching condition \eqref{Hk+} and $J_-(\lam)$ satisfies the matching condition \eqref{Hk-} at the non-smooth point $(S_\infty, 0)$. 
A similar proof shows that the same conclusion holds for $(-S_\infty, 0)$. 
\end{proof}

\bluetext{ The following  propositions stating Lipschitz continuity results can be obtained by the methods of \cite{LPS15, Liu17}, in particular Propositions 3.2 and 3.3 of
\cite{LPS15}. The exclusion of the disk $|\lam| < R$ implies that $|\alpha(\lam)|$ is strictly positive so division by $\alpha$ does not affect the estimates.
Proposition \ref{lip prop},  formulas \eqref{J} -- \eqref{J-}, and the factorizations \eqref{factor-J1} and \eqref{factor-J3}
  imply  also the Lipschitz continuity of the scattering data (Proposition \ref{post lip prop}).}
\begin{proposition}
\label{lip prop}
The maps 
\begin{align*}
q 	         & \mapsto \left. \rho \right|_{\R_\infty} \in H^{2,2}(\R_\infty), &   
q 			& \mapsto n_{21}^-(x_0,\dotarg) \in H^2(\Gamma_\infty^+) \\
q 			& \mapsto n_{12}^-(x_0,\dotarg) \in H^2(\Gamma_\infty^-), &
q 			& \mapsto \frac{1}{\balpha} \in H^2(\Gamma_\infty^+), \\
q 			& \mapsto \frac{1}{\balpha_0} \in H^2(\Gamma_\infty^+), &
q 			& \mapsto 1/\alpha \in H^2(\Gamma_\infty^-),\\
q 			& \mapsto 1/\alpha_0 \in H^2(\Gamma_\infty^-)
\end{align*}
are locally Lipschitz continuous from $H^{2,2}(\R)$ into the respective ranges. Moreover, the map
$$q_{x_0} \mapsto \rho_0 \in H^{2}(\R)$$
is locally Lipschitz continuous from $H^{0,2}(\bbR)$ to $H^{2,0}(\bbR)$.
\end{proposition}




\begin{proposition}
\label{post lip prop}
The maps
$$
\begin{array}{rll}
q &	\mapsto	J_1^\pm(\lam) - I 	&\in H^{2,2}(\dee \Omega_2)\\[3pt]
q &	\mapsto	 J_2(\lam)-I 			&\in H^2(\dee \Omega_3)\\[3pt]
q &	\mapsto J_3^\pm(\lam) - I 	&\in H^{1,1}(\dee \Omega_4)\\[3pt]
q &	\mapsto J_4(\lam)-I 			&\in H^2(\dee \Omega_4)
\end{array}
$$
are locally Lipschitz mappings from $H^{2,2}(\R)$ into their respective ranges.
\end{proposition}

\subsection{Time evolution of the scattering data}
\label{time-evol}
A key property of the inverse scattering method is the simple time evolution of its scattering data.
In \cite{Liu17}, we 
calculated the time evolution of the scattering data where they reduce to a reflection coefficient and  discrete data. 
We need to complement the analysis by examining the time evolution of the jump matrix on the additional section of the contour $\Gamma_\infty$ (see Figure \ref{aug-2}).
 As before, we work in the $\zeta$ variable and carry out the change of variable $\zeta\to\lambda$. Given $\bfM^+(x,t;\zeta)=\bfM^-(x,t;\zeta)v_x(t; \zeta)$ where $v_x(t; \zeta)=e^{-i\zeta^2 x\ad\sigma} v(t; \zeta)$, we compute the time derivative
\begin{equation}
\label{RHP.t}
\bfM^+(x,t;\zeta)_t=\bfM^-(x,t;\zeta)_t v_x(t; \zeta)+\bfM^-(x,t;\zeta)v_x(t; \zeta)_t.
\end{equation}
We recall  that
$\bfM^\pm$  are fundamental solutions for the Lax equations
\begin{subequations}
\label{RHP.Lax}
\begin{align}
\label{RHP.Lax.x}
\frac{\dee M}{\dee x}(x,t;\zeta)  
	&	= -i\zeta^2 \ad \sigma (M ) + \zeta Q(x,t) M + P(x,t) M,\\[3pt]
\label{RHP.Lax.t}
\frac{\dee M}{\dee t}(x,t;\zeta) 
	&=-2i\zeta^4 \ad \sigma (M) +  A(x,t;\zeta)M.
\end{align}
\end{subequations}
where  $\sigma$, $P$, $Q$ are given in terms of $q=q(x,t)$  by \eqref{sigma-Q-P} and
\begin{align}
\label{RHP.A}
A(x,t;\zeta) 
	&= 			2 	\zeta^3 \offmat{q}{-\qbar}
					+ 	i\zeta^2 \diagmat{|q|^2}{-|q|^2} 
					+	i\zeta\offmat{q_x}{\qbar_x}\\[5pt]
\nonumber
	&\quad 		+ \frac{i}{4} \diagmat{|q|^4}{-|q|^4} 
					+ \frac{1}{2} \diagmat{-q_x \qbar + q \qbar_x}{-q\qbar_x  + q_x \qbar}.
\end{align}
Taking the limit $x\to +\infty$ in \eqref{RHP.t}, using the normalization of $\bfM^\pm$ at $+\infty$, and using the fact that $\ad\sigma$ is a derivation, we obtain 
$$v_x(t; \zeta)_t=-2i\zeta^4\ad\sigma v_x(t; \zeta).$$ 
Integrating we obtain 
\begin{equation}
\label{v.ev.zeta}
v_x(t;\zeta)=e^{-2i\zeta^4 t\ad\sigma }v_x(0; \zeta)
\end{equation}
or equivalently for $J_x(\lambda)=e^{-i\lambda x\ad\sigma} J(0; \lambda)$

$$
J_x(\lambda, t)=e^{-2i\lambda^2 t\ad\sigma }J_x(\lambda).
$$

The map $(f,t) \mapsto e^{-2i\lam^2 t}f$ is a bounded continuous map from $X \times [-T,T]$ to $X$ for $X=H^{2,2}(\Omega_2)$, $H^{1,1}(\Omega_1)$, $H^2(\Omega_3)$ and $H^2(\Omega_4)$. This map is also  Lipschitz continuous in $X$ uniformly for $f$ in a bounded subset of $X$ and $t \in [-T,T]$, for a fixed $T>0$. 

From Proposition \ref{post lip prop} and these facts, we deduce the following continuity result.

\begin{proposition}
\label{scattering-continuous}
Suppose that $q_0 \in H^{2,2}(\R)$ and that $J(\lam)$ is the scattering data associated to $q_0$. Denote by $J_\pm(\lam,t)$ the matrices $e^{i\lam^2 t\ad \sigma} J_\pm(\lam)$ where $J_\pm(\lam)$ are the factors given in Theorem \ref{thm:scattering data}. For any $T>0$, the maps
\begin{align*}
H^{2,2}(\R) \times [-T,T]	\ni	(q_0,t)	&	\mapsto		J_-(\lam,t) - I	\in 	H^{2,2}(\dee \Omega_2)\\
H^{2,2}(\R) \times [-T,T]	\ni	(q_0,t)	&	\mapsto		J_-(\lam,t) - I \in 		H^2(\dee \Omega_3)\\
H^{2,2}(\R) \times [-T,T]	\ni	(q_0,t)	&	\mapsto		J_+(\lam,t) - I \in 	H^{2}(\dee \Omega_4)\\
H^{2,2}(\R) \times [-T,T]	\ni	(q_0,t)	&	\mapsto		J_+(\lam,t) - I \in 	H^{1,1}(\dee \Omega_1)
\end{align*}
are all continuous, and uniformly Lipschitz in $q_0$ for $t \in [-T,T]$ and $q_0$ in a bounded subset of $H^{2,2}(\R)$. 
\end{proposition}

\subsection{Auxiliary scattering matrix}
\label{aux-scatt}

In section \ref{sec:scattering-zeta}, we have chosen $x_0\in \bbR$ such that the cut-off potential   $q_{x_0}=q\chi_{ (x_0, \infty )  }$ satisfies the  smallness condition \eqref{x0-condition}. 
By increasing $x_0$ if necessary,
we assume  $\widetilde{q}_{x_0}=q\chi_{ (- \infty, -x_0 )}$ also satisfies \eqref{x0-condition}.
Let $\widetilde{\bfN }$ be constructed in the same way as  $\bfN$ (see \eqref{bfN}) but with potential   ${q}_0$ changed to  $ \widetilde{q}_0$ with normalization at $x\to-\infty$. We define the auxiliary matrix $s$:
\begin{equation}
\label{auxiliary}
s(\lambda)=e^{ix\lambda\ad\sigma}\widetilde{\bfN}^{-1}(x,\lambda)\bfN(x,\lambda).
\end{equation}
For $\lambda\in \Omega_1\cup\Omega_2$
\begin{equation}
\label{delta-alpha}
s(\lambda)=\twomat{\delta(\lambda)^{-1}}{0}{0}{\delta(\lambda)} ,\;\;\;  \;\; \delta(\lambda)=\begin{cases}
\balpha(\lambda) &\quad  \text{Im} \lambda>0\\
\alpha(\lambda)  &\quad  \text{Im}\lambda <0 
\end{cases}
~.
\end{equation}
The jump matrix $\widetilde{J}$ for $\widetilde{ \bfN }$ is obtained from $J$ by conjugation, as
$\widetilde{J}=s_-^{-1}J s_+.$
In analogy with Theorem \ref{thm:scattering data}, we have:
\begin{theorem}
\label{scattering data tilde}
The matrix $\widetilde{J}=s_-^{-1}J s_+$ admits a triangular factorization $\tildJ(\lambda)=\tildJ_-^{-1}(\lambda) \tildJ_+(\lambda)$ where:
\begin{enumerate}
\item[(i)]
$\tildJ_+(\lambda)-I\in H^{2,2}(\partial\Omega_1)$, $\tildJ_{-}(\lambda)-I\in H^{1,1}(\partial \Omega_2) $, $\tildJ_{-}(\lambda)-I\in H^{2}(\partial \Omega_3) $ and
$\tildJ_{+}(\lambda)-I\in H^{2}(\partial \Omega_4)$.
\smallskip
\item[(ii)] $\tildJ_+\restriction_{\partial\Omega_1}-I$ and $\tildJ_-\restriction_{\partial\Omega_3}-I$ are strictly upper triangular while $\tildJ_-\restriction_{\partial\Omega_2}-I$ and $\tildJ_+\restriction_{\partial\Omega_4}-I$ are strictly lower triangular.
\end{enumerate}
\end{theorem} 
\begin{remark}
The reason for working with the Beals-Coifamn solutions normalized at $-\infty$ is to obtain the desired decay in $x$ at $-\infty$. The basic idea is to guarantee that the Fourier variable $|\xi|\geq |x|$. See \cite[Lemma 2.3]{Zhou98} for details.
\end{remark}

%
%

\section{Unique Solvability of the RHP}
\label{sec:unique-solv}
The goal of this section is to prove the unique solvability of the Riemann-Hilbert problem \ref{RHP.lambda} on the contour $\Gamma = \R \cup \Gamma_\infty$ shown in Figure \ref{aug-2}.

\begin{RHP}
\label{RHP.lambda}
Fix $x \in \bbR$. Find a row vector-valued  function 
$\bfN(x,\dotarg)$ on $\C \setminus \Gamma$ with the following properties:
\begin{enumerate}
\item[(i)](Analyticity) $\bfN(x, z)$ is an analytic function of $z$ for $z\in \mathbb{C}\setminus\Gamma $,
\item[(ii)] (Normalization) ${\bfN}(x,z)=(1 ,0)+\bigO{z^{-1}}$ as $z \rightarrow\infty$,
and
\item[(iii)] (Jump condition) For each $\lambda\in\Gamma $, $\bfN$ has continuous 
boundary values
$\bfN_{\pm}(\lambda)$ as $z \rarr \lambda$ from $\Omega_\pm$. 
Moreover, the jump relation 
$$\bfN_+(x,\lambda)=\bfN_-(x,\lambda) J_x(\lambda)$$ 
holds, where 
\begin{align*}
J_x(\lambda)
	&=	e^{-i\lambda x \ad \sigma} 
		\begin{cases}
				\ttwomat{ 1+\lambda|\rho(\lambda)|^2}{\rho(\lambda)}{\lambda\overline{\rho(\lambda)}}{1},
				&	\lambda\in \R_\infty
				\\	
				\\
				\ttwomat{1} {-\rho_0(\lambda)}{-\lambda\overline{\rho_0(\lambda)}}  { 1+\lambda|\rho_0(\lambda)|^2},
				&	\lam \in(-S_\infty, S_\infty)
				\\ \\
				\ttwomat{1}{0}{e^{-2ix_0\lambda}  \dfrac{n^{-}_{21} (x_0, \lambda) }{\balpha(\lambda) \balpha_0(\lambda)  }}{1}
				& \lambda \in \Gamma_\infty^+, \\
				\\
				\ttwomat{1}{-e^{2ix_0\lambda}\dfrac{n^{-}_{12} (x_0, \lambda) }{\alpha(\lambda) \alpha_0(\lambda)  }}{0}{1} 
				&	\lambda \in  \Gamma_\infty^-.
		\end{cases}
\end{align*}

\end{enumerate}
\end{RHP}
\begin{definition}
\label{def-null-lambda}
We say that  the row vector-valued function $\bfN(x,z)$ is a \emph{null vector} for RHP \ref{RHP.lambda} if $\bfN(x,z)$ satisfies (i) and (iii) above but $\bfN(x,z) = \bigO{z^{-1}}$ as $|z| \rarr \infty$. 
\end{definition}

The scattering data that determine the jump matrix $J$ are the functions
$$
SD =  
\left(
	\rho, \rho_0, \alpha, \alpha_0, \balpha_0, n_{12}^-(x_0,\dotarg), n_{21}^-(x_0,\dotarg)
 \right).
$$
 Although these functions are not independent, for the purpose of proving existence and uniqueness of solutions to RHP \ref{RHP.lambda} we may consider them so. Recalling \eqref{R-infty}, we seek a Banach space $Y_0$, consisting of functions $\rho: \bbR_\infty\to\bbC$, with the following properties:
\begin{itemize}
\item[(a)]	There is an injection $i: H^{2,2}(\R_\infty) \rarr Y_0$ that maps bounded subsets of $H^{2,2}(\R)$ to precompact subsets of $Y_0$
\item[(b)]	For each $\rho \in Y_0$, $(1+|\dotarg|) \rho(\dotarg) \in L^2(\R_\infty) \cap L^\infty(\R_\infty)$. 
\item[(c)]	Each $\rho \in Y_0$ is a continuous function with $\lim_{\lam \rarr \infty} \lam \rho(\lam) = 0$. This will allow uniform rational approximation of $(\dotarg)\rho(\dotarg)$ in $L^\infty$.
\end{itemize}

Consider the weighted Sobolev spaces 
$$ H^{\alpha,\beta}(\R) = \left\{ f \in L^2(\R): \, \langle \xi \rangle^{\alpha} \widehat f(\xi), \langle x \rangle^\beta f \in L^2(\R) \right\}.$$
and recall that for any $\eps>0$, $H^{1/2+\eps,0}(\R) \subset C_0(\R)$, where $C_0(\R)$ denotes the continuous functions vanishing at infinity.
Also, recall that the embedding $i: H^{\alpha,\beta}(\R) \hookrightarrow H^{\alpha',\beta'}(\R)$ is compact for $\alpha>\alpha'$ and $\beta>\beta'$. 
From the estimates
$$ 
\norm[L^2]{\langle \dotarg \rangle \rho(\dotarg)} \leq \norm[H^{0,1}(\R)]{\rho}, \quad
\norm[H^{1,0}(\R)]{\langle \dotarg \rangle \rho(\dotarg)}  \leq  \norm[H^{2,2}(\R)]{\rho}
$$
it follows by interpolation that for any $\eps>0$, 
$$
\norm[H^{1/2+\eps,0}(\R)]{\langle \dotarg \rangle \rho (\dotarg)} \leq \norm[H^{1+2\eps,3/2 +\eps}]{\rho}.
$$
 $Y_0= H^{1+2\eps,3/2+\eps}(\R_\infty)$,
 is the image of the fractional Sobolev space $H^{1+2\eps,3/2+\eps}(\R)$ 
under the restriction map $f \mapsto \left. f \right|_{\R_\infty}$. This space satisfies the required properties (a), (b), (c) above.

\begin{definition}
\label{def-space-Y}
We denote by $Y$ the Banach space of scattering data $SD$ with $\rho \in Y_0$ and all other data in $H^1$.
\end{definition}

\begin{remark}
\label{rem-H1-Y}
Note that, for $SD \in Y$, the entries of $J$ all belong to $L^2 \cap L^\infty$.
\end{remark}

Let $Z_0 = H^{2,2}(\R_\infty)$. By Proposition \ref{lip prop}, the range of the direct scattering map actually lies in the following stronger space:

\begin{definition}
\label{def-space-Z}
We denote by $Z$ the set of scattering data $SD$ with $\rho \in Z_0$ and all other data in $H^2$. 
\end{definition}

We choose to consider $SD$ in the larger space in order to obtain uniform resolvent estimates for scattering data in bounded subsets of $Z$ later by a continuity-compactness argument (see Appendix \ref{app-cc}).
We will exploit the fact that, under the natural continuous embedding of $Z$ in $Y$, bounded subsets of $Z$ are identified with precompact subsets of $Y$.  We will prove:

\begin{theorem}
\label{thm:lambda-unique}
Suppose that the scattering data $J(\lam)$ are given by \eqref{J}--\eqref{J-} with $SD \in Y$.  Then RHP \ref{RHP.lambda} has a unique solution for each $x_0 \in \R$. 
\end{theorem}

Following the pattern of the uniqueness result in \cite{JLPS17a,Liu17}, we will prove the existence and uniqueness of solutions in the following way. First, we show that RHP \ref{RHP.lambda} is equivalent to a Fredholm integral equation (the Beals-Coifman integral equation, \eqref{SIE-nu}, for an unknown function $\nu(x,\dotarg)$ on $\Gamma$.  By the Fredholm alternative, it suffices to show that the corresponding homogeneous equation, \eqref{SIE-nu-homo}, has only the trivial solution. In order to do so, in Subsection \ref{Mapping}, we derive similar integral equations associated to an equivalent Riemann-Hilbert Problem, RHP \ref{RHP.zeta}, on the contour $\Sigma$. These integral equations involve an unknown function $\mu$; the inhomogeneous equation is \eqref{SIE-mu} and the homogeneous equation is \eqref{SIE-mu-homo}.  We can use Zhou's uniqueness theorem to show that RHP \ref{RHP.zeta} is uniquely solvable, or, equivalently, \eqref{SIE-mu-homo} has only the trivial solution. Finally, we show that any solution $\nu$ of the homogeneous equation  \eqref{SIE-nu} induces a solution of \eqref{SIE-mu-homo}, It then follows from explicit formulae connecting $\nu$ and $\mu$ that $\nu=0$, establishing the Fredholm alternative for the \emph{original} Beals-Coifman equation \eqref{SIE-nu}.

\subsection{RHPs and  singular integral equations}
We now derive  the Beals-Coifman integral equation for RHP \ref{RHP.lambda}.  The unique solvability of RHP \ref{RHP.lambda}  is equivalent to the unique solvability of 
its associated integral equation.
We  define the nilpotent matrices  $W_x^+$ and $W_x^-$  {in the various parts of the contour} as
$$J_x(\lambda)= (J_{x-})^{-1} J_{x+}=(I-W_x^-)^{-1}(I+ W_x^+) $$
and the Beals-Coifman solution
\begin{equation}
\label{nu-def.}
\nu	=	\bfN^+ (I+ W_x^+)^{-1} = \bfN^-(I-W_x^-)^{-1}
\end{equation}
where
$$ 
(W_x^+,W_x^-) = 
\begin{cases}
	\left( 
		\lowmat{ \lam \overline{ \rho(\lambda) } e^{2i\lam x}}, 
		\upmat{\rho(\lam)e^{-2i\lam x}}
	\right), & \lambda\in\R_\infty,\\
	\\
	\left(
		\upmat{-\rho_0(\lambda)  e^{-2i\lam x}}, 
		\lowmat{-\lambda\overline{\rho_0(\lambda)}  e^{2i\lam x}  }
	\right),
	& \lambda\in(-S_\infty, S_\infty),	\\
	\\
	\left(
		\lowmat{e^{2i\lam x} S_1(\lambda)},
		\lowmat{e^{2i\lam x} S_2(\lambda)}
	\right),
	& \lambda\in\Gamma_\infty^+,		\\
	\\
	\left(
		\upmat{e^{-2i\lam x} S_3(\lambda)},
		\upmat{e^{-2i\lam x} S_4(\lambda)}
	\right),
	& \lambda\in\Gamma_\infty^-.
\end{cases}
$$
The coefficients $S_i(\lambda)$, $i=1,\cdots, 4$,  are not explicitly determined. Only the sums $S_1(\lambda) + S_2(\lambda)$ and $S_3(\lambda) + S_4(\lambda)$  identify to the entries $(2,1)$ and $(1,2)$ of the jump matrix $J_x(\lambda)$ respectively, in the corresponding part of the contour.
If $SD \in Y$, then
$W_x^\pm$ in $L^\infty \cap L^2$, while if $SD \in Z$, $W_x^\pm \in H^1$. 

We can write  the Beals-Coifman solution $\nu(x,\lambda)$ explicitly in terms of the Jost functions.  \bluetext{ From \eqref{nu-def.}, we have}  two equivalent formulas.
\begin{align*}
\nu(x,\lambda)
	&=	\begin{cases}
				\left( \dfrac{n_{11}^-(x,\lambda)}{\ba(\lambda)} \quad n_{12}^+(x,\lambda) \right)		
				\ttwomat{1}{0}{- e^{2i\lambda x}\lam \overline{\rho(\lam)}}{1} \\
				\\
				\left(n_{11}^+(x,\lambda) \quad \dfrac{n_{12}^-(x,\lambda)}{\alpha(\lambda)} \right) 		
				\ttwomat{1}{\rho(\lam) e^{-2i\lam x}}{0}{1}
			\end{cases}
			&\lambda \in \R_\infty\\
			\\
\nu(x,\lambda)
	&=	\begin{cases}
 				\left( \dfrac{n_{11}^-(x,\lambda)}{\ba(\lambda)} \quad n_{12}^+(x,\lambda) \right)	
 				\ttwomat{1}{0}{- e^{2i\lambda x} S_1(\lambda)}{1}	 \\
 				\\
				\left(   N^{(2 )}_{11+}(x,\lambda)   \quad     N^{(2 )}_{12+}(x,\lambda) \right) 
				 \ttwomat{1}{0}{e^{2i\lam x} S_2(\lambda)}{1} 
			\end{cases}
			&\lambda \in \Gamma_\infty^+\\
			\\
\nu(x,\lambda)
	&=	\begin{cases}	
				\left(   N^{(2 )}_{11+}(x,\lambda)   \quad     N^{(2 )}_{12+}(x,\lambda) \right)
				\ttwomat{1}{\rho_0(\lambda)  e^{-2i\lam x}}{0}{1}\\
				\\
 				\left(   N^{(2 )}_{11-}(x,\lambda)   \quad     N^{(2 )}_{12-}(x,\lambda) \right) 
 				\ttwomat{1}{0}{-\lambda\overline{\rho_0(\lambda)}  e^{2i\lam x}   }{1} 
 		\end{cases}
		&\lambda \in (-S_\infty, S_\infty)\\
		\\
\nu(x,\lambda)
	&=	\begin{cases}
 				\left(   N^{(2 )}_{11-}(x,\lambda)   \quad     N^{(2 )}_{12-}(x,\lambda) \right) 	
 				 \ttwomat{1}{-e^{-2i\lam x} S_3(\lambda)}{0 }   { 1}	 \\
 			\\
			\left(n_{11}^+(x,\lambda) \quad \dfrac{n_{12}^-(x,\lambda)}{\alpha(\lambda)} \right)  
			\ttwomat{ 1}{e^{-2i\lam x} S_4(\lambda) } {0 } {1 }
			\end{cases}
			&	\lambda \in \Gamma_\infty^-
\end{align*}
From (\ref{nu-def.}),  we have
$$\bfN^+ - \bfN^-=\nu\left(W_x^+ + W_x^- \right).$$
The Plemelj formula and the normalization condition (ii) in RHP \ref{RHP.lambda} provide the  Beals-Coifman integral equation:
\begin{align}
\label{SIE-nu}
\nu(x,\lam)	
	&=		\left( 1, 0 \right)+\left(\calC_{W_x}\nu\right)(\lam)
		\end{align}
where 
$$	\calC_{W_x}\nu =	C^+_{\Gamma}(\nu W^-_x)+
		C^-_{\Gamma}(\nu W^+_x).$$
RHP \ref{RHP.lambda} is equivalent to the  integral equation  \eqref{SIE-nu}  \cite[Proposition 3.3]{Zhou89-1}.
Similarly, if $\bfN$ is a null vector for RHP \ref{RHP.lambda} in the sense of Definition \ref{def-null-lambda}
and $\nu$ is defined in \eqref{nu-def.},  we have
\begin{equation}
\label{SIE-nu-homo}
\nu(x,\lam) = 
	\calC_{W_x}\nu(\lam).
\end{equation}
\bluetext{
If  $SD \in Y$, eq.  \eqref{SIE-nu} (resp.\ \eqref{SIE-nu-homo}) is seen as  an  integral equation for $\nu -1\in L^2(\Gamma)$ (resp.\ $\nu \in L^2(\Gamma)$), while if $SD \in Z$, 
it is an integral equation  for $\nu -1 \in  H^1(\Gamma)$ (resp.\ $\nu \in  H^1(\Gamma)$).
}

For $\lambda\in \R_\infty$, \eqref{SIE-nu} reads:
\begin{align*}
\nu_{11}(x,\lam) 	
	&=	1 +\int_{\R_\infty}
			\frac{\nu_{12}(x,s) s \overline{\rho(s)} e^{2is x}}{s-\lambda+i0} \frac{ds}{2\pi i} 
			-   \int_{-S_\infty}^{S_\infty} \frac{\nu_{12}(x,s) s \overline{\rho_0(s)} e^{2is x}}{s-\lam} \frac{ds}{2\pi i} \\
\nonumber
	&	\qquad + \int_{\Gamma_\infty^+}  \frac{\nu_{12}(x,s) (S_1(s)+S_2(s))e^{2is x}}{s-\lam} \frac{ds}{2\pi i}    
\end{align*}
\begin{align*}
\nu_{12}(x,\lam) 	
	&=	 \int_{\R_\infty} 
					\frac{\nu_{11}(x,s) \rho(s) e^{-2is x}}{s-\lam-i0} \frac{ds}{2\pi i} 
			-  \int_{-S_\infty}^{S_\infty} \frac{  \nu_{11}(x,s)\rho_0(s) e^{-2is x}  }{ s-\lambda }\dfrac{ds}{2\pi i} \\
\nonumber
	& \qquad
			+ \int_{\Gamma_\infty^-} \frac{\nu_{11}(x,s) (S_3(\lambda)+S_4(\lambda)) e^{-2is x}}{s-\lam} \frac{ds}{2\pi i}  .
\end{align*}
The integral equations for $\lambda\in (-S_\infty, S_\infty)$ and $\lambda\in \Gamma_\infty^\pm$  are obtained analogously. 
The solution to Problem \ref{RHP.lambda}
is given,  in terms of $\nu = (\nu_{11}, \nu_{12})$ by
\begin{align}
\label{N.recon}
\bfN(x,z) 	&= 	(1,0)+ \frac{1}{2\pi i}\int_{\Gamma} \frac{\nu(x,s)(W_x^+(s) + W_x^-(s) )}{s-z} \, ds.	
\end{align}
The goal is an existence and uniqueness  result for  solution to Problem \ref{RHP.lambda}.  To make use of the symmetry relations of the jump conditions  and Zhou's vanishing lemma, we need to consider the  equivalent RHP  in the $\zeta$ variable with jump contour $\R \cup i\R\cup \Sigma_\infty$ given by Figure \ref{aug-1}.

\begin{RHP}
\label{RHP.zeta}
Fix $x \in \bbR$. Find a matrix-valued  function 
 $\bfM(x,\dotarg)$ with the following properties:
\begin{enumerate}
\item[(i)](Analyticity) $\bfM(x, z)$ is a  $2\times 2$ matrix-valued  analytic function of $z$ for $z\in \C \setminus \Sigma$  where the contour $\Sigma$ is given by Figure \ref{aug-1}.

\item[(ii)] (Normalization) 
	$${\bfM}(x,z)=\twomat{1}{0}{0}{1}+\mathcal{O}(z^{-1}) \text{ as } z \rightarrow\infty.$$

\item[(iii)] (Jump condition) For each $\zeta\in\Sigma$, $\bfM$ has continuous 
boundary values
$\bfM_{\pm}(\lambda)$ as $z \to \zeta$ from $\Omega_\pm$. 
Moreover, the jump relation 
$$\bfM_+(x,\zeta)=\bfM_-(x,\zeta)e^{-ix\zeta^2\ad\sigma} v(\zeta)$$ 
holds, where $v(\zeta)$ is given by \eqref{v}-\eqref{v-circle--}.
\end{enumerate}
\end{RHP}

\begin{definition}
\label{def-null-zeta}
We say that a matrix-valued function $\bfM(x,z)$ is a \emph{null vector} for RHP \ref{RHP.zeta} if $\bfM(x,z)$ satisfies (i) and (iii) above and $\bfM(x,z) = \bigO{z^{-1}}$ as $|z| \rarr \infty$. 
\end{definition}

Observe that, given scattering data $SD$ for RHP \ref{RHP.lambda} in the space $Y$ from Definition \ref{def-space-Y}, the induced scattering data for RHP \ref{RHP.zeta} consist of bounded continuous functions, square-integrable on the unbounded contours. Thus RHP \ref{RHP.zeta} is well-defined with the $\bigO{z^{-1}}$ condition replaced by an $L^2$-condition on $\bfM_\pm - I$ 
(and the condition $\bfM_\pm \in L^2$ for an $L^2$-null vector).

\begin{proposition}
\label{vanishing-zeta}
The only $L^2$ null vector for RHP \ref{RHP.zeta} with scattering data induced from $SD \in Y$is the zero vector.
\end{proposition}
\begin{proof}
The proof is a  direct consequence of Proposition \ref{schwarz invariance} and \cite[Theorem 9.3]{Zhou89-1}. 
\end{proof}

It is useful to formulate  Proposition \ref{vanishing-zeta} in terms of the homogeneous Beals-Coifman equation \bluetext{ associated to}  RHP \ref{RHP.zeta}, which we now derive.

The jump matrix  $v_x(\zeta)$ admits the following factorization
$$v_x(\zeta)=(1-w_x^-)^{-1}(1+w_x^+) .$$
We set 
\begin{equation}
\label{mu-def.}
\mu	=	\bfM^+ (1+ w_x^+)^{-1} = \bfM^-(1-w_x^-)^{-1}.
\end{equation}
In analogy with  RHP \ref{RHP.lambda}, the  Beals-Coifman integral equation for RHP \ref{RHP.zeta} is
\begin{equation}
\label{SIE-mu}
\mu=I+\calC_{w_x}\mu=I+ C^+_{\Sigma}(\mu w^-_x)+C^-_{\Sigma}(\mu w^+_x)
\end{equation}
where $I$ is the $2\times 2$ identity matrix.
If $\bfM$ is a null vector in the sense of Definition \ref{def-null-zeta} and $\mu$ is defined by \eqref{mu-def.},
then
\begin{equation}
\label{SIE-mu-homo}
\mu=\calC_{w_x}\mu.    
\end{equation}
We can now reformulate Proposition \ref{vanishing-zeta} as follows:

\begin{proposition}
\label{vanishing-zeta-mu}
Assume that $w^\pm$ are obtained from scattering data $SD$ in $Y$. Then, the only solution to \eqref{SIE-mu-homo} in $L^2(\Sigma)$ is the zero vector. 
\end{proposition}

\subsection{A Mapping Between Null Spaces}
\label{Mapping}
To complete the proof of Theorem \ref{thm:lambda-unique}, we show that any solution $\nu$ of \eqref{SIE-nu-homo}
induces a solution $\mu$ of \eqref{SIE-mu-homo} and that if $\mu=0$, then $\nu=0$. For notational brevity, we suppress
the dependence of $\mu$ and $\nu$ on $x$, which remains fixed throughout the discussion.

\begin{lemma}
\label{lemma:nu.to.mu}
Suppose that $W_x^\pm$ are generated from scattering data $SD \in Y$.  
For $\nu=(\nu_{1}, \nu_{2})$ a solution of the homogeneous Beals-Coifman equation (\ref{SIE-nu-homo}) {in $L^2(\Gamma)$}, define
\begin{equation}
\label{nu.to.mu-con}
\mu(x,\zeta)
		=\ttwomat	{\mu_{11}(x,\zeta)}
						{\mu_{12}(x,\zeta)}
						{\mu_{21}(x,\zeta)}
						{\mu_{22}(x,\zeta)} 
		=\ttwomat	{\nu_{11}(x,\zeta^2)}
						{\zeta \nu_{12}(x,\zeta^2)}
						{-\zeta \overline{\nu_{12}(x,\zetabar^2)}}
						{\overline{\nu_{11}(x,\zetabar^2)}}.
\end{equation}
Then $\mu \in L^2(\Sigma)$ solves \eqref{SIE-mu-homo}. 
\end{lemma}

\begin{remark}
\label{rem:mu-to-nu}
We can invert \eqref{nu.to.mu-con} to recover $\nu$ via the formulas
\begin{equation}
\label{mu.to.nu-con}
\nu_{11}(x,\zeta^2) = \mu_{11}(x,\zeta), \quad \nu_{12}(x,\zeta^2) = \mu_{12}(x,\zeta)/\zeta. 
\end{equation}
In particular, if $\mu = 0$, then $\nu=0$.
\end{remark}                  
\begin{proof}
\emph{Define} a matrix-valued function $\mu$ by \eqref{mu.to.nu-con} for a given solution $\nu$ of eq. \eqref{SIE-nu-homo}.
It is easy to see that
$$\mu_{11}(x,-\zeta)=\mu_{11}(x,\zeta) , \quad  \mu_{12}(x,-\zeta) =-\mu_{12}(x,\zeta).  $$
In \cite[Lemma 5.2.2]{Liu17} we have shown that for $\nu\in L^2(\Gamma)$ and $\rho\in Y_0$,  
$$\mu_{11}(x,\zeta)r(\zeta)=\nu_{11}(x,\zeta^2)\zeta\rho(\zeta^2), \,\, \mu_{12}(x,\zeta)\br(\zeta) =\zeta \nu_{12}(x,\zeta^2)\zeta\overline{ \rho(\zeta^2)} $$ 
are both square-integrable on the part of the $\Sigma$ contour outside the circle $\Sigma_\infty$. Thus $\mu w_x^\pm$ is an $L^2$ function on $\Sigma$. 
Once \bluetext{ \eqref{SIE-mu-homo} is obtained from  \eqref{SIE-nu-homo}},
 $\mu\in L^2(\Sigma)$ follows from the boundedness of Cauchy projection on $L^2$-functions.

In \cite[Chapter 5]{Liu17}, the second author established the transition  from \eqref{SIE-nu-homo} to  \eqref{SIE-mu-homo} when $\Gamma=\bbR$ and $\Sigma=\bbR\cup i\bbR$.  Thus, we  only consider   the contour integrals:
\begin{equation}
\label{Sigma+}
 I^+ := \int_{\Gamma_\infty^+}  \frac{\nu_{12}(x,s) (S_1(s)+S_2(s))e^{2is x}}{s-\lam} \frac{ds}{2\pi i}    
 \qquad
\end{equation}
and 
\begin{equation}
\label{Sigma-}
I^-:= \int_{\Gamma_\infty^-} \frac{\nu_{11}(x,s) (S_3(\lambda)+S_4(\lambda)) e^{-2is x}}{s-\lam} \frac{ds}{2\pi i}  
\end{equation}
Let $\lambda=\zeta^2$ and fix the branch $[0, 2\pi)$. 
\begin{align*}
I^+ &=\int_{\Gamma_\infty^+}  \frac{\nu_{12}(x,\lambda) n^{-}_{21} (x_0, \lambda) e^{2i\lambda( x- x_0)}}{(\lambda-\lam_0)\balpha(\lambda)\balpha_0(\lambda)}  \frac{d\lambda}{2\pi i}  
 \\
&
=\int_C \dfrac{\zeta^{-1} \mu_{12}(x,\zeta) \zeta m^{-}_{21} (x_0,\zeta) e^{2i\zeta^2( x- x_0)}}{(\zeta^2-\zeta^2_0)\ba(\zeta)\ba_0(\zeta)}  \dfrac{d\zeta^2}{2\pi i} \\
&=\int_C \left(  \dfrac{ \mu_{12}(x,\zeta) m^{-}_{21} (x_0,\zeta) e^{2i\zeta^2( x- x_0)}}{(\zeta-\zeta_0)\ba(\zeta)\ba_0(\zeta)} 
-\dfrac{ \mu_{12}(x,\zeta) m^{-}_{21} (x_0,\zeta) e^{2i\zeta^2( x- x_0)}}{(-\zeta-\zeta_0)\ba(\zeta)\ba_0(\zeta)}    \right)\dfrac{1}{2\zeta}\dfrac{d\zeta^2}{2\pi i} \\
&=\int_C \dfrac{ \mu_{12}(x,\zeta) m^{-}_{21} (x_0,\zeta) e^{2i\zeta^2( x- x_0)}}{(\zeta-\zeta_0)\ba(\zeta)\ba_0(\zeta)} \dfrac{d\zeta}{2\pi i} 
 - \int_C \dfrac{ \mu_{12}(x,\zeta) m^{-}_{21} (x_0,\zeta) e^{2i\zeta^2( x- x_0)}}{(-\zeta-\zeta_0)\ba(\zeta)\ba_0(\zeta)}  \dfrac{d\zeta}{2\pi i} \\
 &=I^+_1+I^+_2
\end{align*}
Setting $\lambda=R^2 e^{i\theta}$ with $\theta \in (\pi,0)$, 
we  integrate $I^+_1$   over the arc $\zeta={R} e^{i\eta}$ where $\eta$ goes from $\pi/2$ to $0$. For $I^+_2$, we make the change of variable $\zeta \to-\zeta$, then  using of the oddness of $\mu_{12}$ and $m_{21}$. We obtain 
$$ - \int \dfrac{ \mu_{12}(x,\zeta) m^{-}_{21} (x_0,\zeta) e^{2i\zeta^2( x- x_0)}}{(-\zeta-\zeta_0)\ba(\zeta)\ba_0(\zeta)}  \dfrac{d\zeta}{2\pi i} = \int \dfrac{ \mu_{12}(x,\zeta) m^{-}_{21} (x_0,\zeta) e^{2i\zeta^2( x- x_0)}}{(\zeta-\zeta_0)\ba(\zeta)\ba_0(\zeta)}  \dfrac{d\zeta}{2\pi i}.$$
For $I^+_2$, we integrate over the arc ${R}e^{i\eta}$ with $\eta$ going from $3\pi/2$ to $\pi$. This completes the change of variable for \eqref{Sigma+}. Similarly, 
 we have 
\begin{align*}
I^-&=-\int_{\Gamma_\infty^-}  \frac{\nu_{11}(x,\lambda) n^{-}_{12} (x_0, \lambda) e^{2i\lambda( x_0- x)}}{(\lambda-\lam_0)\alpha(\lambda)\alpha_0(\lambda)}  \frac{d\lambda}{2\pi i} 
 \\
&
=		-\int_C  
			\frac{\mu_{11}(x,\zeta) m^{-}_{12} (x_0, \zeta) e^{2i\zeta^2( x_0- x)}}{\zeta(\zeta^2-\zeta_0^2)a(\zeta) a_0(\zeta)}  
		\frac{d\zeta^2}{2\pi i}  \\
& =	-\int_C 
			\dfrac{1}{2\zeta^2} \left( \dfrac{1}{\zeta-\zeta_0}+
			\dfrac{1}{\zeta+\zeta_0} \right) \frac{\mu_{11}(x,\zeta) m^{-}_{12} (x_0, \zeta) e^{2i\zeta^2( x_0- x)}}{a(\zeta) a_0(\zeta)}  
		\frac{d\zeta^2}{2\pi i} \\
& =	-\int_C 
				\frac{\mu_{11}(x,\zeta) m^{-}_{12} (x_0, \zeta) e^{2i\zeta^2( x_0- x)}}{\zeta_0(\zeta-\zeta_0)a(\zeta) a_0(\zeta)}  
		\frac{d\zeta}{2\pi i} \\
& \quad 
		+ 
		\int_C 
			\frac{\mu_{11}(x,\zeta) m^{-}_{12} (x_0, \zeta) e^{2i\zeta^2( x_0- x)}}{\zeta_0(\zeta+\zeta_0)a(\zeta) a_0(\zeta)}  
		\frac{d\zeta}{2\pi i} \\
&=I^-_1+I^-_2.
\end{align*}
We write  $\lambda=R^2 e^{i\theta}$ with $\theta \in (\pi, 2\pi)$. For $I^-_1$, we  integrate over the arc $\zeta={R} e^{i\eta}$ where $\eta$ goes from $\pi/2$ to $\pi$. For $I^-_2$, we make the change of variable $\zeta\to -\zeta$, then making use of the evenness and  oddness of $\mu_{11}$ and $m_{12}$  respectively to obtain 
\begin{align*}
\int \dfrac{ \mu_{11}(x,\zeta) m^{-}_{12} (x_0,\zeta) e^{2i\zeta^2( x_0-x)}}{\zeta_0(\zeta+\zeta_0)\ba(\zeta)\ba_0(\zeta)}  \dfrac{d\zeta}{2\pi i} &=- \int \dfrac{ \mu_{11}(x,\zeta) m^{-}_{12} (x_0,\zeta) e^{2i\zeta^2( x_0-x)}}{\zeta_0(-\zeta-\zeta_0)\ba(\zeta)\ba_0(\zeta)}  \dfrac{d\zeta}{2\pi i}\\
                                                                                                                  &=- \int \dfrac{ \mu_{11}(x,\zeta) m^{-}_{12} (x_0,\zeta) e^{2i\zeta^2( x_0-x)}}{\zeta_0(\zeta-\zeta_0)\ba(\zeta)\ba_0(\zeta)}  \dfrac{d\zeta}{2\pi i}.
\end{align*}
For $I^-_2$, we integrate over the arc ${R}e^{i\eta}$ where $\eta$ goes from $3\pi/2$ to $2\pi$. 
Integrals involving $\mu_{21}$ and $\mu_{22}$ are derived using complex conjugations.
\end{proof}

\begin{proof}[Proof of Theorem \ref{thm:lambda-unique}]
First, for scattering data $SD \in Y$, the operator $(I-\calC_{W_x})$ is a Fredholm operator on $L^2(\Gamma)$. This follows from \cite[Lemma 2.60]{TO16} since we allow uniform rational approximation of the function  $(\dotarg)\rho(\dotarg)$ for $\rho\in Y_0$. Next, we claim that  $\ker_{L^2(\Gamma)} (I - \calC_{W_x})$ is trivial. If $\nu \in \ker_{L^2(\Gamma)} (I - \calC_{W_x})$, then
by Lemma \ref{lemma:nu.to.mu}, $\nu$ induces a vector $\mu \in \ker_{L^2(\Sigma)}(I-\calC_{w_x})$, which must be the zero vector by Proposition \ref{vanishing-zeta-mu}. It follows from Remark \ref{rem:mu-to-nu} that $\nu=0$. Finally,  from Fredholm theory,  $(I-\calC_{W_x})$ is invertible in $L^2(\Gamma)$, which is equivalent to unique solvabilty of RHP \ref{RHP.lambda}.
\end{proof}


\begin{corollary}
\label{prop:lambda-resolvent}
The resolvent $(I-\calC_{W_x})^{-1}$ exists for all $x \in \R$ and all $SD \in Y$.
\end{corollary}

%
%

\section{Mapping Properties of the Inverse Scattering Map}
\label{sec:inverse}

Recall that the potential $q$ is reconstructed by solving the "right" Riemann-Hilbert problem \ref{RHP.lambda}  (for a solution normalized as $x \to +\infty$). As shown in Section \ref{aux-scatt},  the "left" Riemann-Hilbert problem (with jump matrix characterized by Theorem \ref{scattering data tilde}) can be conjugated to the "right" Riemann Hilbert problem, so we concentrate on the mapping properties of the reconstruction from the right. We omit the (standard) proof that the left and right reconstructions agree, as well as the proof that the inverse map composed with the direct map is the identity map on $H^{2,2}(\R)$. Thus, in the statements of Theorems \ref{thm:L^2-decay} and \ref{thm:L^2-smooth}, an assertion is made about the reconstructed potential on $\R$, but details of the proof are only given for the restriction of $q$ to a half-line of the form $(c,\infty)$.

We start  with the reconstruction formula for the potential $q$ from  given scattering data $J_\pm$ as characterized  in Theorem \ref{thm:scattering data}:
\begin{align}
\label{recon}
q(x) &=\left( -\dfrac{1}{\pi}\int_{\Gamma}\nu(x,\lambda)\left(W_x^+(\lambda) + W_x^- (\lambda) \right)d\lambda \right)_{12}\\
\nonumber
       &=\left( -\dfrac{1}{\pi}\int_{\Gamma}\nu(x,\lambda) 
       e^{-i\lambda x \ad\sigma} \left(J_+(\lambda) - J_- (\lambda) \right)d\lambda \right)_{12}
\end{align}
where the $12$ subscript denotes the second entry of the row vector, and $\Gamma$ is the contour shown in Figure \ref{aug-2}.

Let  $\mathbf{A}(\Omega_\pm)$ denote the space of analytic functions in the region $\Omega_\pm$  of the complex plane and  $R(\partial \Omega_\pm )$ the space of functions whose restrictions on $\partial\Omega_\pm$ are rational.  
Following  a reduction technique of \cite{Zhou98}, we  construct functions $\omega_\pm\in\bfA( \Omega_\pm )$ such that, for $k=2$:
\begin{enumerate}
\item[1.] $\omega_\pm \in R(\partial \Omega_\pm )$ and  $\omega_\pm-I=O(z^{-2})$ as $z\rightarrow \infty$.
\item[2.] $\omega_\pm$ has the same triangularity as $J_\pm$, and
\item[3.] $\omega_\pm(z)=J_\pm(z)+o((z-a)^{k-1})$ for 
$a = \pm S_\infty$.
\end{enumerate}

The construction of $\omega_\pm$  is given in \cite[Appendix I]{Zhou89-2}. For example,  consider the approximation of $J_-\restriction_{ \partial \Omega_2}$ . 
Since $(J_- -I)\restriction _{\partial \Omega_2}$ is in $H^2$, we construct a rational function $\omega_-$ such that $(\omega_--J_-)\restriction \partial \Omega_2$ vanishes at $\pm S_\infty$ to order  1. Explicitly 
$$\omega_-(\pm S_\infty)-I=\rho(\pm S_\infty), \qquad \omega'_-(\pm S_\infty)=\rho'(\pm S_\infty).$$
This is performed by the following steps:
\begin{enumerate}
\item[(i).]  Choose $z_0 \notin \overline{\Omega}_2$ and denote $p_{\pm}$ the Taylor polynomial of degree 1 of $(z-z_0)^n\rho(z)$ at $z=\pm S_\infty $.We choose $n\geq 6$.
\item[(ii).] By \cite[Lemma A1.2]{Zhou89-2}, there is  a polynomial $p(z)$ of degree at most 3 such that
$$p(z) -  p_\pm(z) = O(z \mp S_\infty)^2.$$
\item[(iii).] Set $\omega_-(z)= (z-z_0)^{-n}p(z)$. Clearly, $\omega_-(z)-\rho(z)$   vanishes at $\pm S_\infty$ to  order  1. Since $n\geq 6$,  $\omega_-\in H^{2,2}(\partial\Omega_2)$ and $\omega$ is analytic in $\Omega_2$.
\end{enumerate}

By construction,
\begin{equation}
\label{J.to.calJ}
J=\omega_-^{-1}(J_-\omega_-^{-1})^{-1}(J_+\omega_+^{-1})\omega_+
	\equiv 	\omega_-^{-1}\mathcal{J}_-^{-1}\calJ_+\omega_+
	\equiv 	\omega_-^{-1} \calJ \omega_+.
\end{equation}
The advantage of working with $\calJ$ is that  
$\calJ_\pm-I$ vanishes at  $\pm S_\infty$
to 
order $1$ :
\begin{equation}
\label{calJ}
\calJ_\pm(\lambda)=I+o((\lambda-a)^{1}), \quad a=\pm S_\infty.
\end{equation}
The continuity of $\calJ_\pm$ and its derivative at $\lambda=\pm S_\infty$ is a key point to perform the decay estimates for the reconstructed potential.
Notice that $\calJ_\pm$ are defined like  $J_\pm$ in Theorem \ref{thm:scattering data}  and they will be used when establishing estimates such as \eqref{upper-c} and \eqref{lower-c}.
For   $x\geq0$, $\calJ_x$ is the jump \bluetext{ matrix} 
for the  RHP
$$\mathcal{N}_+(x,\lambda)=\mathcal{N}_-(x,\lambda ) \calJ_x(\lambda) \quad \lambda\in \Gamma$$
if and only if  $J_x$ is the jump  \bluetext{ matrix} 
 for the RHP \ref{RHP.lambda}. Here $\bfN=\mathcal{N} e^{-i\lambda x\ad\sigma}\omega$ where $e^{-i\lambda x\ad \sigma}\omega\in \bfA L^{\infty}(\mathbb{C}\setminus\Gamma)\cap \bfA L^2(\bbC \setminus\Gamma)$   is  guaranteed by  construction. 
Note that  $\bfN$ and $\mathcal{N}$ give rise to the same $\nu$, solution of the associated Beals-Coifman equation.
 The potential is given by
$$q(x)=\left( -\dfrac{1}{\pi}\int_{\Gamma}\nu(x,\lambda) 
       e^{-i\lambda x \ad\sigma} \left(\calJ_+(\lambda) - \calJ_- (\lambda) \right)d\lambda \right)_{12}.$$
Due to the large $z$-behavior of $\omega_\pm(z)$, we have
$$ \left( \lim_{z\to\infty} 2iz N(x, z)\right)_{12}= \left( \lim_{z\to\infty} 2iz \mathcal{N}(x, z) \right)_{12} ,$$
which 
shows that   $\omega$ gives no contribution to the reconstruction of $q$ for $x\geq 0$.  We may  thus as well work with $\calJ$. 

The next step consists in augmenting
 the contour as in  Figure~\ref{fig:new.mod.contour} below. The newly modified contour is denoted $\Gamma_m$.
 The advantage of $\Gamma_m$ is that it reverses the orientation of the segment $ (S_\infty^-, S_\infty^+)$ and thus allows 
to prove usual estimates of the Cauchy projections when the contour is restricted to $\bbR$.
The added  (dashed) contours have no effect on the RHP since the jump matrices there are chosen to be the identity.
  
\begin{figure}[H]
\caption{The newly  modified contour $\Gamma_m$}

\vskip 0.5cm

\begin{tikzpicture}
\draw[thick,->-]	(-6,0) 	-- 	(-2,0);
\draw[thick,->-]	(-2,0)			arc(180:0:2);
\draw[thick,->-]	(-2,0)	--	(2,0);
\draw[thick,->-]	(-2,0)			arc(180:360:2);
\draw[thick,->-]	(2,0)	--	(6,0);
\draw[thick,dashed,->-, domain=-2:2] plot ({-\x}, {(4 - \x*\x)/4});
\draw[thick,dashed,->-, domain=-2:2] plot ({-\x}, {-(4 - \x*\x)/4});
\draw[black,fill=black]		(-2,0)		circle(0.06cm);
\draw[black,fill=black]		(2,0)		circle(0.06cm);
\node[below] at (-2.75,0)	{$(-S_\infty,0)$};
\node[below] at (2.55,0)	{$(S_\infty,0)$};
\node[above]	 at (3.5,0)				{$+$};
\node[below]	 at (3.5,0)				{$-$};
\node[above]  at (1.15,0)				{$+$};
\node[below]	 at (1.15,0)				{$-$};
\end{tikzpicture}
\label{fig:new.mod.contour}
\end{figure}
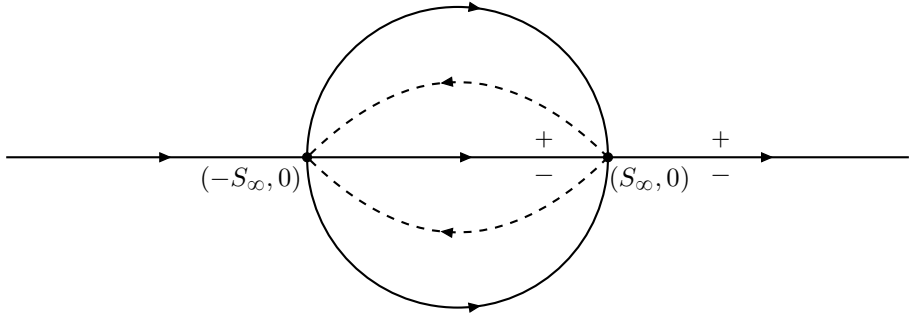
We redefine $\calJ_\pm$ as follows:
\begin{enumerate}
\item[1.] $\calJ_\pm=I$ on the added  (dashed) contours, 
\item[2.] $\calJ_\pm(\lambda)$ is the lower/upper triangular factor in the lower/upper triangular factorization of $\calJ$ ($\calJ^{-1}$) on $\bbR$ for $|\lambda|>|S_\infty|$, ($|\lambda|<|S_\infty|$) and
\item[3.] for $\lambda\in \Gamma_\infty$, $\calJ_\pm(\lambda)=I$ for $\imag \lambda\lessgtr 0$ and $\calJ_\pm(\lambda)=\calJ(\lambda)$ for $\imag \lambda\gtrless 0$.
\end{enumerate}
The newly defined $\calJ_\pm$ satisfy all properties
\bluetext{listed in}  Theorem \ref{thm:scattering data}.  To analyze the scattering map, we will use the revised reconstruction formula
\begin{equation}
\label{q.recon.gamma-m}
q(x)=\left( -\dfrac{1}{\pi}\int_{\Gamma_m}\nu(x,\lambda) 
       e^{-i\lambda x \ad\sigma} \left(\calJ_+(\lambda) - \calJ_- (\lambda) \right)d\lambda \right)_{12}.
\end{equation}
We will at first suppress dependence of the scattering data on $t$ (subsections \ref{sec:recon-decay} and \ref{sec:recon-smooth}), but recall it again in subsection \ref{sec:recon-time}.

Associated to the Riemann-Hilbert problem with jump matrix $\calJ_x$ is a Beals-Coifman integral equation where the Beals-Coifman operator is given by 
\begin{equation}
\label{C.calJx}
C_{\calJ_{x}  }\phi=
	C^+\left[ \phi(\calJ_{x+} -I)\right]+	
	C^-\left[\phi(I-\calJ_{x-})\right].
\end{equation}

Throughout the analysis we will use the following uniform resolvent bound.

\begin{proposition}
\label{prop:resolvent}
Suppose that $\calJ=\omega_-J\omega_+^{-1}$ where $\calJ$ has the form of \eqref{J}-\eqref{J-},  is constructed from scattering data in a bounded  
subset $B$ of $Z$ (see Definition \ref{def-space-Z}), and $\calJ$ admits an algebraic factorization $\calJ=\calJ_-^{-1}\calJ_+$ where $\calJ_\pm$  have \bluetext{the same triangularities  as in} 
Theorem \ref{thm:scattering data}.  Then, for fixed $a\in\bbR$ and all $x\geq a$ the estimate
\begin{equation}
\label{bdd resolvent}
\sup_{\calJ \in B} \norm[L^2 \to L^2]{(I-C_{\calJ_x})^{-1}} < \infty
\end{equation}
holds.
Finally, the map 
$$ \calJ \mapsto \left( x \mapsto (I-C_{\calJ_x})^{-1} \right) $$
is Lipschitz continuous into the space $C\left([a,\infty); \scrB(L^2 )\right)$.
\end{proposition}

\begin{proof}
We check hypotheses (i)--(iii) of Proposition \ref{prop:cc} with $X=L^2(\R)$, $Y$ as given in Definition \ref{def-space-Y}, and $Z$ as given in Definition \ref{def-space-Z}.

\smallskip

(i) The continuity of the map $(\calJ,x) \mapsto C_{\calJ_x}$ and the uniform continuity estimate follow immediately from \eqref{C.calJx}. 

\smallskip

(ii) The proof of Corollary \ref{prop:lambda-resolvent} applies with no essential change to show that $(I-C_{\calJ_x})^{-1}$ exists for all $x \in \R$ and $\calJ \in Y$.

\smallskip
(iii)  To prove this estimate we need to show that $(I-C_{\calJ_x})^{-1}$ is bounded as $x \to +\infty$ for each fixed $\calJ$. To do this we use a standard parametrix construction and 
approximation argument due to Zhou  \cite{Zhou89-1}. 

Let $\brevcalJ_\pm=(\calJ_\pm)^{-1}$. A standard computation shows that
$$
I-T_{ \calJ_x}=( I-C_{\calJ_{x}  } )(I-C_{\brevcalJ_{x}  })
$$
where 
$$
T_{ \calJ_x} \phi =
	C^+\left( C^-\phi(\calJ_{x+} -\calJ_{x-}) \right)
		\left(I-\brevcalJ_{x-}\right) +
	C^-\left( C^+\phi(\calJ_{x+} -\calJ_{x-}) \right)
		\left(\brevcalJ_{x+} -I\right).
$$
The operator $T_{\calJ_x}$ is compact so that $(I-C_{\brevcalJ_{x}})$ is a Fredholm regulator for $(I-C_{\calJ_x})$.  It suffices to show that 
$$ \lim_{x \to +\infty} \norm[L^2 \to L^2]{T_{\calJ_x}} = 0 $$
since
$$ (I-C_{\calJ_x})^{-1} = (I-C_{\brevcalJ_x}) (I - T_{\calJ_x})^{-1} $$
and $\norm[L^2 \to L^2]{(I-C_{\brevcalJ_x})}$ is bounded uniformly in $x$.
By mimicking the proof of \cite[Theorem 6.1]{Zhou89-1} (with some sign changes since we consider the limit $x \to +\infty$ rather than $x \to -\infty$) we can show by rational approximation that, for \emph{fixed} $\calJ$, $\norm[L^2 \to L^2]{T_{\calJ_x}} \to 0$ as $x \to -\infty$. Taking $b$ so that $\norm[L^2 \to L^2]{T_{\calJ_x}} < 1/2$ for $x \geq b$, we obtain a uniform bound on $\norm[L^2 \to L^2]{(I-C_{\calJ_x})^{-1}}$ for $x \geq b$. Since $x \mapsto (I-C_{\calJ_x})^{-1}$ is continuous, this implies that $\sup_{x \geq a} \norm[L^2 \to L^2]{(I - C_{\calJ_x})^{-1}}$ is bounded for any $a \in \R$. 

We can now apply Proposition \ref{prop:cc} to obtain the uniform bound  and the asserted Lipschitz continuity.

\end{proof}

\subsection{Decay property of the reconstructed potential}
\label{sec:recon-decay}
Denote
$$H^{0,2}(\R)=\left\{ q \in L^2(\R):  x^2 q(x) \in L^2(\R) \right\}.$$
\begin{theorem}
\label{thm:L^2-decay}

If $J$ is given as in Theorem \ref{thm:scattering data} and $q$ is defined by \eqref{q.recon.gamma-m},
then $q\in H^{0,2}(\R)$. 
Moreover, the map from data $\calJ=\calJ_-^{-1}\calJ_+$ defined  in \eqref{J.to.calJ} and  obeying the hypothesis of Theorem \ref{thm:lambda-unique}, to 
\bluetext{  $q \in H^{0,2}(\R)$} is Lipschitz continuous.
\end{theorem}

\begin{definition}
\label{contour-Gamma}
 Define the  subsets of the contour $\Gamma_m$:
$$
\Gamma_\pm :=\bbR \cup \left( \lbrace \imag \lambda \gtrless 0 \rbrace \cap \Gamma_m\right),
$$
$$
\Gamma' :=\Gamma_\infty^\pm=   \text{either}  \; \Gamma_\infty \cap \lbrace \imag \lambda \gtrless 0 \rbrace \, \, \text{or}  \, \,\bbR.
$$
\end{definition}

\begin{lemma}
\label{lemma:decay-est}
(See \cite[Lemma 2.9]{Zhou98}) For $x\geq 0$,
\begin{align}
\label{C+}
\left\Vert  C^+_{ \Gamma' \to\Gamma_+  }  (I-\calJ_{x-})  \right\Vert_{L^2} &\leq \dfrac{c}{1+x^2}\norm[H^2]{\calJ_- -I}\\
\label{C-}
\norm[L^2]{ C^-_{ \Gamma'\to\Gamma_-  }  (I-\calJ_{x+}) }&\leq \dfrac{c}{ (1+x^2 )^{1/2}}\norm[H^1]{\calJ_+-I}\\
\label{upper-c}
\norm[L^2(\Gamma_\infty^+)]{\calJ_{x+} - I   } &\leq\dfrac{c}{(1+x^2)^{1/2}} \norm[H^1]{\calJ_+ - I   } \\
\label{lower-c}
\norm[L^2(\Gamma_\infty^-)]{\calJ_{x-} - I   } &\leq\dfrac{c}{1+x^2} \norm[H^2]{\calJ_- - I   } \\
\label{C11}
\left( \norm[L^2(\Gamma)]{ (  C_{\calJ_{x}  }  )^2 I} \right)_{11}&\leq \dfrac{c}{1+x^2}\norm[H^1]{\calJ_+-I}\norm[H^2]{\calJ_- -I} \\
\label{C22}
\left( \norm[L^2(\Gamma)]{ (  C_{\calJ_{x}  }  )^2 I} \right)_{22}&\leq \dfrac{c}{(1+x^2)^{1/2}}\norm[H^1]{\calJ_+-I}\norm[H^2]{\calJ_- -I}.
\end{align}
\end{lemma}

\begin{proof}[Proof of Theorem \ref{thm:L^2-decay}]
Proposition \ref{prop:resolvent} and Lemma \ref{lemma:decay-est}
provide the tools for estimating the decay of the potential $q$, recalling that $\nu$ appearing in \eqref{recon} is equal to
$(I-C_{\calJ_x^\pm})^{-1} (1,0) $.   
We 
\bluetext{ decompose the following integral into the sums of 4 integrals:}
\begin{equation}
\label{recon integral}
\int  \left( \left(  I -  C_{\calJ_{x}  } \right)^{-1} I \right) e^{-i\lambda x\ad\sigma}(\calJ_+-\calJ_-)d\lam=\int_1+\int_2+\int_3+\int_4
\end{equation}
where the  integrals on the right-hand-side are defined in \eqref{int-1}--\eqref{int-4}.    We  extract information on $q(x)$ from the (1-2) entry. Here and thereafter, the integral sign without subscripts refers to an integral taken  on the entire contour  displayed in   Figure  \ref{fig:new.mod.contour}. We write 
\begin{equation}
\label{int-1}
\int_1:=\int_\bbR \left( \calJ_{x+}-\calJ_{x-} \right)+\int_{\Gamma_\infty^+}\left(  \calJ_{x+}-I\right)+\int_{\Gamma_\infty^-}\left( I- \calJ_{x-}\right).
\end{equation}
Notice that $\calJ_x^- -I$ is strictly upper triangular on $\Gamma$ and $\calJ^- -I$ is in $H^2$ so we conclude that the (1-2) entry of the integral above is in $H^{0,2}$ by mapping properties of the Fourier transform and \eqref{lower-c}. The second integral
\begin{equation}
\label{int-2}
\int_2:=\int \left( C_{\calJ_{x}} I \right) (\calJ_{x+} -\calJ_{x-})
\end{equation}
is zero on the (1-2) entry, it  thus 
makes
no contribution to the reconstruction of $q$. For the third integral,
\begin{align}
\label{int-3}
\int_3 &:=\int \left(  ( C_{\calJ_{x}})^2 I \right) (\calJ_{x+} -\calJ_{x-})\\
\nonumber
          &=\int \left( C^+\left(C^-\left( \calJ_{x+}-I \right)  \right) (I-\calJ_{x-}) \right)\left( \calJ_{x+}-I \right)\\
          \nonumber
          &\quad+ \int \left( C^-\left(C^+\left( I-\calJ_{x-} \right)  \right) (\calJ_{x+}-I) \right)\left(I- \calJ_{x-}\right).
\end{align}
The (1-2) entry is 
\begin{align*}
\int_{\Gamma_\infty^-} &\left( C^-_{\Gamma_+ \to \Gamma}\left(C^+_{\Gamma \to \Gamma_+}\left( I-\calJ_{x-} \right)  \right) (\calJ_{x+}-I) \right)\left(I- \calJ_{x-}\right)\\
&+\int_\bbR \left( C^-_{\Gamma^+ \to \Gamma}\left(C^+_{\Gamma \to \Gamma_+}\left( I-\calJ_{x-} \right)  \right) (\calJ_{x+}-I) \right)C^+_\bbR \left(I- \calJ_{x-}\right),
\end{align*}
and from \eqref{C+} and \eqref{lower-c},  we conclude that 
$$\left\vert \left( \int_3 \right)_{12} \right\vert \leq \dfrac{c}{(1+x^2)^2}.$$ Finally we set $$g=(1-C_{\calJ_{x}  } )^{-1} \left(  (  C_{\calJ_{x}  }  )^2 I \right)   $$
and write
\begin{equation}
\label{int-4}
\left\vert \, \int_4 \, \right\vert : = \left\vert \int \left[  \left( C^+g (I-\calJ_{x-}) \right)\left( \calJ_{x+}-I \right)  +\left( C^- g \left( \calJ_{x+}-I \right)  \right)\left(I- \calJ_{x-}\right)  \right]  \right\vert.
\end{equation}
Again,  the (1-2) entry is given by 
\begin{align*}
\int_{\Gamma_\infty^-} \left( C^- g \left( \calJ_{x+}-I \right)  \right)\left(I- \calJ_{x-}\right) 
+\int_\bbR  \left( C^-_{\Gamma^+\to\bbR} ~ g \left( \calJ_{x+}-I \right)  \right) C_\bbR^+ \left(I- \calJ_{x-}\right) 
\end{align*}
and from \eqref{C+}, \eqref{lower-c},  \eqref{C11}, and Proposition \ref{prop:resolvent}  we conclude that 
$$\left\vert \left( \int_4 \right)_{12} \right\vert \leq \dfrac{c}{(1+x^2)^2}.$$
The estimate for $x\in  (-\infty, a)$ is obtained by considering the RHP with jump condition described in Theorem \ref{scattering data tilde}. 
Lipschitz continuity of the map follows from Proposition
\ref{prop:resolvent} and equation \eqref{recon integral}.
\end{proof}

\subsection{Smoothness property of the reconstructed potential}
\label{sec:recon-smooth}
\begin{theorem}
\label{thm:L^2-smooth}
If the jump matrix $J$ is given by  Theorem \ref{thm:scattering data}  and $q$ is defined by \eqref{q.recon.gamma-m},
then $q\in H^{2,0}(\R)$. Moreover, the map from data $\calJ$ defined as in \eqref{J.to.calJ} and obeying the hypothesis of Theorem \ref{thm:lambda-unique} to 
\bluetext{$q \in H^{2,0}(\R)$} is Lipschitz continuous. 
\end{theorem}

In order to study smoothness properties of the reconstructed potential, we first show that the functions $M_\pm$ solving RHP \ref{RHP.zeta} solve a differential equation in the $x$ variable. It follows that the same is true of the solution $\mu$ of \eqref{SIE-mu} since $\mu$ is obtained from either $M_+$ or $M_-$ through postmultiplication by a matrix 
of the form $e^{-ix\zeta^2 \ad \sigma} A(\zeta)$. We can then change variables to find a differential equation in $x$ obeyed by the (matrix-valued) solution $\nu$ of RHP \ref{RHP.lambda}. 

\begin{proposition}
\label{prop:recon}
The functions $M_\pm$ obey the differential equation
\eqref{LS.m} where $M$, $P$ and $Q$ are constructed from the solution $\mu$ of \eqref{SIE-mu} as follows:
\begin{align*}
M(x,\zeta)&=I+\int_{\Sigma} \dfrac{\mu(x,s)\left( w_x^+(s)+w_x^-(s)  \right)}{s-\zeta}\dfrac{ds}{2\pi i}\\
Q(x) 	&=	-\frac{1}{2\pi} \ad \sigma  
						\left( \int_{\Sigma} 
									\mu(x,\zeta) 
										\left(
											w_x^+(\zeta)+ w_x^-(\zeta)
										\right) 
								\, d\zeta
						\right)\\
P(x)	&=	Q(x) i (\ad \sigma)^{-1} Q(x). 
\end{align*}
\end{proposition}
The proof of the above proposition is a slight modification of the proof of Proposition 5.3.1 in \cite{Liu17}. Here we need take into account of the integration along the additional circle $\Sigma_\infty$. 
\begin{proof}[Proof of Theorem \ref{thm:L^2-smooth}]
We first notice that $\mu$ given by \eqref{mu-def.} solves the linear problem \eqref{LS.m}:
$$
\frac{d}{dx}\mu = \left( -i\zeta^2 \ad \sigma + \zeta Q(x) + P(x)\right)\mu.
$$
 We now use the change of variable $\zeta\to\lambda$ 
 to obtain 
$$
\frac{d}{dx}\nu	=	\left( -i\lambda \ad \sigma   +\offmat{q}{-\lam \qbar} + P \right)\nu.
$$
We further write
\begin{multline}
\label{nu.de.ad}
\dfrac{d}{dx} \left( \nu e^{-i\lambda x \ad\sigma} \left(\calJ_+ - \calJ_-  \right) \right)	=	\\
\left( -i\lambda \ad \sigma   +\offmat{q}{-\lam \qbar} + P \right)   \left( \nu  e^{-i\lambda x \ad\sigma} \left(\calJ_+- \calJ_- \right)  \right).
\end{multline}
Unlike RHP \ref{RHP.lambda} in which $\nu$  appears as a row vector, here $\nu$ is a $2\times 2$ matrix:
$$  \nu =\Twomat{\nu_{11}(x,\lambda)}
									{ \nu_{12}(x,\lambda)}
									{-\lambda \overline{\nu_{12}(x,\overline{\lambda}) } }
									{\overline{ \nu_{11}(x,\overline{\lambda} ) } } $$
and its first row $(\nu_{11}, \nu_{12})$ is the solution to RHP \ref{RHP.lambda}.
We  integrate both sides  of \eqref{nu.de.ad} along the contour  shown in Figure \ref{fig:new.mod.contour} 
\begin{multline*}
\dfrac{d}{dx} \int \left( \nu e^{-i\lambda x \ad\sigma} \left(\calJ_+ - \calJ_-  \right) \right)	=\\
		\int 
			\left( -i\lambda \ad \sigma   +\offmat{q}{-\lam \qbar} +P \right)   
			\left( \nu  e^{-i\lambda x \ad\sigma} 
				\left(\calJ_+- \calJ_- \right)  
			\right).
\end{multline*}
The potential $q$ is given by the (1-2) entry of this matrix form integral.
Using  that $\calJ_-\in H^{2,2}(\Gamma)$,  the (1-2) entry of 
$$\int -i\lambda\ad\sigma \left( \nu  e^{-i\lambda x \ad\sigma} \left(\calJ_+- \calJ_- \right) \right)$$
is an $L^2$-function of $x$, following the same argument  as in the proof of Theorem \ref{thm:L^2-decay}. To show that the (1-2) entry of 
$$
\int \left( \offmat{q}{-\lam \qbar}
+P \right)   \left( \nu  e^{-i\lambda x \ad\sigma} \left(\calJ_+- \calJ_- \right)  \right)
$$
is an $L^2$-function of $x$, we  use that $q\in L^2\cap L^\infty$ which comes from the fact that 
$|q|\leq {c}/{(1+x^2 )^2}$  shown in Theorem \ref{thm:L^2-decay}.
This proves that $q_x\in L^2$.  To estimate  $q_{xx}$, we differentiate  \eqref{nu.de.ad} with respect to $x$. 
Explicitly, we have that 
$$
\dfrac{d^2}{dx^2} \int \left( \nu e^{-i\lambda x \ad\sigma} \left(\calJ_+ - \calJ_-  \right) \right) =\int_1+\int_2
$$
where
$$\int_1:=\int \left( -i\lambda \ad \sigma   +\offmat{q}{-\lam \qbar} + P \right)^2 \left( \nu e^{-i\lambda x \ad\sigma} \left(\calJ_+ - \calJ_-  \right) \right).  $$
and 
$$ \int_2 :=\int \left( \offmat{q}{-\lam \qbar}_x + P_x \right)\left( \nu e^{-i\lambda x \ad\sigma} \left(\calJ_+ - \calJ_-  \right) \right)   $$
 Again following the previous argument and using that $\calJ_-\in H^{2,2}$ and $q\in H^1$, we conclude that $q_{xx}\in L^2 (-a, +\infty)$. A similar argument using scattering data given by Theorem \ref{scattering data tilde} and solving the corresponding RHP shows that $q_{xx}\in L^2(-\infty, a)$. 
Lipschitz continuity of the map follows from the uniform boundedness of the resolvent operator given by \eqref{bdd resolvent} and  Proposition \ref{prop:cc}. This completes the proof of Theorem \ref{thm:L^2-smooth}.
\end{proof}

\subsection{Time Evolution of the Reconstructed Potential}
\label{sec:recon-time}

 We now recall the explicit time-dependence of $q(x,t)$ on $t$ through the law of evolution \eqref{v.ev.zeta}, and write the reconstruction formula as
$$
q(x,t)=\left( -\dfrac{1}{\pi}\int_{\Gamma_m}\nu(x,\lambda,t) 
       e^{it\theta(x,t,\lam)} \left(\calJ_+(\lam) - \calJ_- (\lam) \right)d\lam \right)_{12}	
$$
where $\Gamma_m$ is the contour shown in Figure \ref{aug-2} and
$$ \theta(x,t,\lam) =-2 \lam^2 - (x/t)\lambda. $$
To study the time-evolution of $q(x,t)$, it will be convenient to work in the $\zeta$ variable
and write
\begin{equation}
\label{q.recon.t}
q(x,t) =\left( -\dfrac{1}{\pi}\int_{\Sigma_m}\mu(x,\zeta,t) 
       e^{it\theta(x,\zeta^2,t)}\left(\widetilde{\calJ_+}(\zeta) - \widetilde{\calJ_-} (\zeta) \right)d\zeta \right)_{12}
\end{equation}
where $\Sigma_m$, shown in Figure~\ref{fig:Gamma-m}, is the inverse image of $\Gamma_m$ under the map $\zeta\mapsto \lambda=\zeta^2$, 
$\widetilde{J^\pm}$ are the scattering data corresponding to $\calJ^\pm$ under the change of variables
and $\mu$ solves the Beals-Coifman integral equation corresponding to the scattering data $\widetilde{J^\pm}$.

\begin{figure}[H]
\caption{The Modified Contour $\Sigma_m$}

\medskip

\begin{tikzpicture}[scale=0.5]
\draw[thick,->-]	(0,4)		--	(0,0);
\draw[thick,->-]	(0,0)		--	(-4,0);
\draw[thick,->-]	(0,0)		--	(4,0);
\draw[thick,->-]	(0,-4)		--	(0,0);
\draw[thick,->-]	(0,-6)		--	(0,-4);
\draw[thick,->-]	(-4,0)		--	(-6,0);
\draw[thick,->-]	(4,0)		--	(6,0);
\draw[thick,->-]	(0,6)		--	(0,4);
\draw[thick,->-]	(0,4)				arc(90:180:4);
\draw[thick,->-]	(0,4)				arc(90:0:4);
\draw[thick,->-]	(0,-4)				arc(270:180:4);
\draw[thick,->-]	(0,-4)				arc(270:360:4);
\draw[dashed,->-,>=stealth]		(4,0)	--	(0,4);
\draw[dashed,->-,>=stealth]		(-4,0)	--	(0,4);
\draw[dashed,->-,>=stealth]		(-4,0)	--	(0,-4);
\draw[dashed,->-,>=stealth]		(4,0)	--	(0,-4);
\end{tikzpicture}
\label{fig:Gamma-m}
\end{figure}
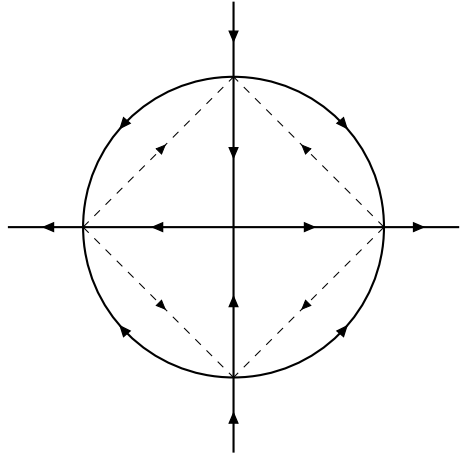

The continuity of the direct and inverse maps implies that, given Cauchy data $q_0 \in H^{2,2}(\R)$, we may approximate  $q_0$ by a sequence $\{ q_n \}$ from $\calS(\R)$ that converges in $H^{2,2}(\R)$ to $q_0$ and obtain a sequence of approximants $q_n(x,t)$ for $q(x,t)$ which converge in $H^{2,2}(\R)$ as $n \to \infty$ uniformly in $t$ in any bounded interval. It follows that, in order to prove that $q(x,t)$ given by \eqref{q.recon.t} is a weak solution of \eqref{DNLS2}, it suffices to assume that $q_0 \in \calS(\R)$ and argue by approximation.

The following proposition can be proved using the same technique used to prove \cite[Proposition 7.0.4]{Liu17}).
\begin{proposition}
\label{prop:RHP.t}
Suppose that  $M_\pm$
solve the RHP \eqref{RHP.zeta}. Let 
\begin{align*}
 Q(x,t) 	&= {  -  } \frac{1}{2\pi} \ad \sigma  
						\left[ \int_{\Sigma_m} \mu(x,\zeta) 
								\left(\widetilde{J^+}(\zeta)  - \widetilde{J^-}(\zeta) \right) d\zeta
						\right]=\twomat{0}{q}{-\qbar}{0}    \\
P(x,t) 	&=	i Q(x,t) (\ad \sigma )^{-1} Q(x,t) = 
						\frac{i}{2} \diagmat{|q|^2}{-|q|^2},
\end{align*}
and $A(x,t)$ given as in \eqref{RHP.A}. 
 Then $M_\pm$  are fundamental solutions of  the Lax equations \eqref{RHP.Lax}.
\end{proposition}

Given a fundamental solution of the Lax equations \eqref{RHP.Lax}, it now follows by a standard argument  \cite[Appendix B]{LPS15}, that $q(x,t)$, defined as the $(1,2)$ entry of $Q(x,t)$, solves the integrable equation \eqref{DNLS2}. Thus:

\begin{proposition}
\label{prop:q-solves}
Suppose that $q_0 \in H^{2,2}(\R)$ and let $J_\pm$ be the corresponding scattering data. Then
$q(x,t)$ defined by \eqref{q.recon.t} solves \eqref{DNLS2}.
\end{proposition}

\subsection{Proof of Theorem \ref{Theorem-main}}

Combining the results of subsections \ref{sec:recon-decay}, \ref{sec:recon-smooth}, and \ref{sec:recon-time} we can now prove the main theorem.

\begin{proof}[Proof of Theorem \ref{Theorem-main}]
Given initial data $q_0 \in H^{2,2}(\R)$, the direct scattering map has the continuity properties asserted in Proposition \ref{post lip prop}, so that the time-evolved scattering data has the continuity properties asserted in Proposition \ref{scattering-continuous}.
By Theorem  \ref{thm:lambda-unique}, RHP \ref{RHP.lambda} is uniquely solvable for each $x,t$, and by Theorems \ref{thm:L^2-decay} and \ref{thm:L^2-smooth}, the map from scattering data $J(\cdot,t)$ to reconstructed potential $q(\cdot,t)$ is Lipschitz continuous into $H^{2,2}(\R)$. The map $(q_0,t) \mapsto q(x,t)$ defined by the composition of the direct scattering map, the flow map, and the inverse scattering map (i) maps $(q_0,0)$ to $q_0$ (by standard arguments which we omit here), (ii) is jointly continuous in $(q_0,t)$ (by the continuity of the direct and inverse maps),  and (iii) is locally Lipschitz continuous in $q_0$ (by the Lipschitz continuity asserted in Theorems \ref{thm:L^2-decay} and \ref{thm:L^2-smooth}). Finally, it follows from Proposition \ref{prop:q-solves} that $q(x,t)$ is a weak solution of \eqref{DNLS2}. 
\end{proof}

\section*{Acknowledgements}

We  thank  Deniz Bilman  and  Peter Miller  for helpful discussions on Zhou's ideas.
This work was supported by a grant from the Simons Foundation/SFARI (359431, PAP).
CS  is supported in part by  Discovery Grant 2018-04536 from the Natural Sciences and Engineering Research Council of Canada.



\appendix

\section{Sobolev Spaces on Self-Intersecting Contours}
\label{app:Sobolev}

In this Appendix, we define the Sobolev spaces $H^k_\pm(\Gamma)$ and $H^k_z(\Gamma)$ needed for the analysis of RHP \ref{RHP.lambda}. These spaces were  introduced by Zhou \cite{Zhou89-1};
see \cite[\S 2.6--2.7]{TO16} for a  \bluetext{ discussion on} 
 their role in the analysis of Beals-Coifman integral equations associated to RHP's.  

If $\Gamma = \Gamma_1 \cup \ldots \cup \Gamma_n$ and the $\Gamma_i$ are either half-lines, line segments, or arcs, the space $H^k(\Gamma)$ consists of  functions $f$ on $\Gamma$ with the property that $\left. f \right|{\Gamma_i} \in H^k(\Gamma_i)$. The space $H^k(\Gamma_i)$ \bluetext{ is well defined} 
since each $\Gamma_i$ can be parameterized by arc length and functions on $\Gamma_i$ viewed as functions on a subset of $\R$.  A function $f \in H^k(\Gamma_i)$ has a representative which is continuous, together with its derivatives $f^{(j)}$ up to order $k-1$. Limits of $f^{(j)}$ at the endpoints of  $\Gamma_i$ are well-defined for $0 \leq j \leq k-1$. The space $H^k_+(\Gamma)$ (resp.\ $H^k_-(\Gamma)$) consists of  the  functions of $H^k(\Gamma)$ which are continuous together with their derivatives up to order $k-1$ along the solid and dashed components shown in Figure \ref{fig-Gamma+-comp} (resp.\ Figure \ref{fig-Gamma--comp}). 

\begin{center}
\begin{figure}[H]
\begin{subfigure}{0.45\textwidth}
\caption{Boundary components of $\Omega^+$}
\bigskip
\begin{tikzpicture}[scale=0.47]
\draw[white]						(-6,0)		--	(-6,6)		--	(6,6)	--	(6,0);
\draw[white,fill=gray!10]		(-6,0)		--	(-3,0)		-- 	(-3,0) arc(-180:0:3)	--	(6,0)	--
										(6,-6)		--	(-6,-6)	--	(-6,0);
\draw[white,fill=gray!10]		(-3,0)	 arc(180:0:3)	--	(-3,0);
\draw[black, fill=black] 		(3,0) 		circle [radius=0.15];
\draw [black, fill=black] 		(-3,0) 	circle [radius=0.15];
\draw[thick,->-]		(-3,0) arc(180:0:3);
\draw[thick,dashed,->-]		(-3,0)	 arc(180:360:3);
\draw [thick,->-] 					(-6,0) -- (-3,0);
\draw [thick,->-,dashed]		(3,0)	--	(-3,0);
\draw [thick,->-] 					(3,0) -- 	(6,0);
\node [below] at (0,-0.5) 		{$\Omega^+$};
\node [above] at (0,0.5) 		{$\Omega^-$};
\node [below] at (4,-3) 		{$\Omega^-$};
\node [above] at (4, 3) 		{$\Omega^+$};
\node[below] at (2.3,0)		{\footnotesize{$S_\infty$}};
\node[below] at (-2.1,0)		{\footnotesize{$-S_{\infty}$}};
\end{tikzpicture}
\label{fig-Gamma+-comp}
\end{subfigure}
\qquad
\begin{subfigure}{0.45\textwidth}
\caption{Boundary components of $\Omega^-$}
\bigskip
\begin{tikzpicture}[scale=0.47]
\draw[white]						(-6,0)		--	(-6,-6)		--	(-6,6)		--	(-6,0);
\draw[white,fill=gray!10]		(-6,0)		--	(-3,0)		-- 	(-3,0) arc(180:0:3)	--	(6,0)	--
										(6,6)		--	(-6,6)		--	(-6,0);
\draw[white,fill=gray!10]		(-3,0)	 arc(-180:0:3)	--	(-3,0);
\draw[black, fill=black] 		(3,0) 		circle [radius=0.15];
\draw [black, fill=black] 		(-3,0) 	circle [radius=0.15];
\draw[thick,->-,dashed]		(-3,0) arc(180:0:3);
\draw[thick,->-]					(-3,0)	 arc(180:360:3);
\draw [thick,->-] 					(-6,0) -- (-3,0);
\draw [thick,->-,dashed]		(3,0)	--	(-3,0);
\draw [thick,->-] 					(3,0) -- 	(6,0);
\node [below] at (0,-0.5) 		{$\Omega^+$};
\node [above] at (0,0.5) 		{$\Omega^-$};
\node [below] at (4,-3) 		{$\Omega^-$};
\node [above] at (4, 3) 		{$\Omega^+$};
\node[below] at (2.3,0)		{\footnotesize{$S_\infty$}};
\node[below] at (-2.1,0)		{\footnotesize{$-S_{\infty}$}};
\end{tikzpicture}
\label{fig-Gamma--comp}
\end{subfigure}
\caption{Boundary Components of $\Omega^\pm$}
\end{figure}
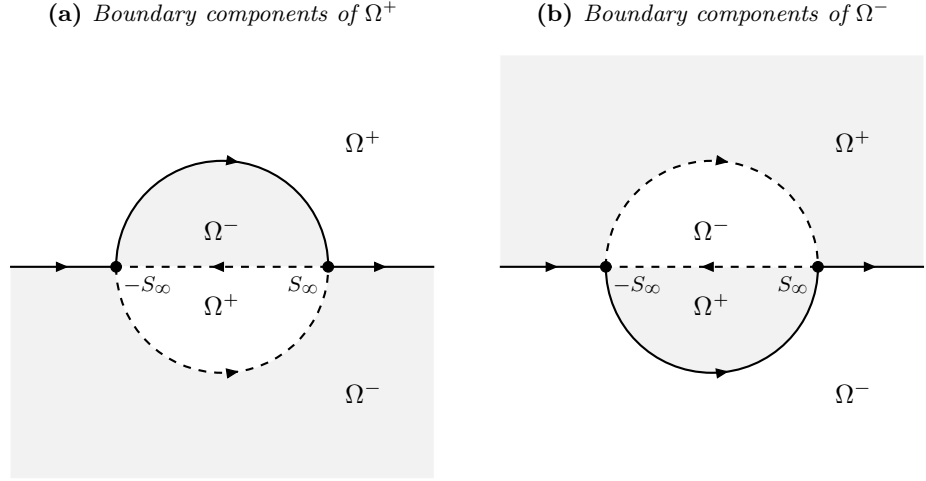
\end{center}
To describe the continuity conditions, let 
$$(f_i^{j})_\pm = \lim_{z \rarr S_{\pm\infty}, z \in \Gamma_i} f^{(j)}(z)$$ 
where the contours $\Gamma_i$ are as shown in Figure \ref{fig-zero-sum} \bluetext{ (a)--(b)}. 
A function $f \in H^k_+(\Gamma)$ obeys the conditions
\begin{equation}
\label{Hk+}
(f^j_1)_- 	= 	(f^j_2)_-,	\quad	(f^j_3)_- 	= 	(f^j_4)_-, \qquad\quad
(f^j_2)_+	=	\bluetext{ (f^j_1)_+},	\quad	(f^j_3)_+	=	(f^j_4)_+
\end{equation}
for $0 \leq j \leq k-1$, where in each case the first condition comes from continuity across the solid contour, and the second from continuity across the dashed contour.
Similarly, a function $f \in H^k_-(\Gamma)$ obeys the conditions
\begin{equation}
\label{Hk-}
(f^j_1)_- 	= 	(f^j_3)_-,	\quad	(f^j_2)_- 	= 	(f^j_4)_-, \qquad\quad
(f^j_3)_+	=	\bluetext{ (f^j_1)_+},	\quad	(f^j_2)_+	=	(f^j_4)_+
\end{equation}
for $0 \leq j \leq k-1$. 

The space $H^k_z(\Gamma)$ consists of those  functions in $H^k(\Gamma)$ which obey the following zero-sum conditions at the two intersection points $\pm S_\infty$.
\begin{equation}
\label{Hk0}
\begin{aligned}
(f_1^{j})_-  + (f_4^{j})_-  - (f_2^{j})_- - (f_3^{j})_- 		&= 0	\\
(f_2^{j})_+  + (f_3^{j})_+  -  (f_4^{j})_+ -  \bluetext{  (f_1^{j})_+	}	& = 0
\end{aligned}
\end{equation}
where the $\pm$ signs are determined by the orientation of the contour as indicated in Figures
\ref{fig-Gamma+-comp} and \ref{fig-Gamma--comp}.

It is easy to see from \eqref{Hk+}, \eqref{Hk-}, and \eqref{Hk0} that $H^k_\pm(\Gamma) \subset H^k_z(\Gamma)$. In \cite[Lemma 2.51]{TO16}, it is shown that if $f \in H^k_z(\Gamma)$, then the Cauchy projectors $C^\pm_\Gamma f \in H^k_\pm(\Gamma)$.  This mapping property is very natural since $C^+_\Gamma f$ (resp.\ $C^-_\Gamma f$) is the boundary value of a function analytic in $\Omega_+$ (resp.\ $\Omega^-$). It  follows 
that $$H^k_z(\Gamma) = H^k_+(\Gamma)  + H^k_-(\Gamma)$$ 
with the decomposition given by  $f = C_\Gamma^+(f) + (-C_\Gamma^-(f))$.

\section{The Continuity-Compactness Argument}
\label{app-cc}

In this appendix, we give the abstract functional-analytic argument needed to prove uniform resolvent estimates required in section \ref{sec:inverse} for the Lipschitz continuity of the inverse scattering map. Proposition \ref{prop:cc} can also be used to simplify proofs of analogous uniform estimates in \cite{JLPS17a,LPS15}. In what follows $\scrB(X)$ denotes the Banach space of bounded operators on the Banach space $X$.

\begin{proposition}
\label{prop:cc}
Let $X$, $Y$, and $Z$ be Banach spaces and suppose that there is a continuous embedding $i:Z \rarr Y$ with the property that bounded subsets of $Z$ map to precompact subsets of $Y$.  Suppose that $C_{J,x}$ is a family of 
bounded
operators on a Banach space $X$ indexed by $J \in Y$ and $x \in \R$. Finally, suppose that:
\begin{itemize}
\item[(i)]		The map $(J,x) \mapsto C_{J,x}$ is continuous as a map
				from $Y \times \R$ into $\scrB(X)$, and the estimate
				$$ \sup_{x \in \R} 
						\norm[\scrB(X)]{C_{J,x} - C_{J',x}} 
				\lesssim 
				\norm[Y]{J-J'}
				$$
				holds,
\item[(ii)]		The resolvent $(I - C_{J,x})^{-1}$ exists for each $x \in \R$ and 
				$J \in Y$, and
\item[(iii)]	For each $J \in Y$, the estimate
				$$ \sup_{x \in [a,\infty)} \norm[\scrB(X)]{(I-C_{J,x})^{-1}} < \infty $$
				holds.
\end{itemize} 
Then for any bounded subset $B$ of $Z$, 
$$\sup_{J \in B} \left( \sup_{x \in [a,\infty)} \norm[\scrB(X)]{(I-C_{J,x})^{-1}} \right) < \infty $$
and the map 
$$ J \mapsto \left\{ x \mapsto (I-C_{J,x})^{-1} \right\} $$
is locally Lipschitz continuous as a map from $Z$ into $C([a,\infty);\scrB(X))$.
\end{proposition}

\begin{remark}
1. In applications, (i) is easy to prove from the explicit form of the Beals-Coifman integral operators, (ii) follows from Fredholm theory and a vanishing theorem for the RHP, and (iii) follows from the continuity of the map $$x \mapsto (I-C_{J,x})^{-1}$$ and the fact that, in the limit $x \rarr \infty$, the integral kernel of the operator is highly oscillatory. 
2. In applications, the bound in hypothesis (iii) is typically only true for half-lines. One can replace $[a,\infty)$ by $(-\infty,a]$ and obtain the same result.
\end{remark}

\begin{proof}
Denote by $C([a,\infty),\scrB(X))$ the Banach space of continuous $\scrB(X)$-valued functions of $x \in [a,\infty)$ equipped with the norm
$$ \norm[C([a,\infty), \scrB(X))]{f} = \sup_{x \in \R} \norm[\scrB(X)]{f(x)}. $$
Consider the map 
\begin{equation}
\label{Y-res}
Y \ni J \mapsto \left(  x \mapsto (I-C_{J,x})^{-1} \right) \in C([a,\infty),\scrB(X)).
\end{equation}
Assumptions (i), (ii), (iii) and the second resolvent formula show that this map is well-defined and continuous. Using the injection $i$ we can identify bounded subsets of $Z$ with precompact subsets of $Y$. We can then use the continuity of the map \eqref{Y-res} to conclude that the image of any bounded subset of $Z$ has compact closure in $C([a,\infty),\scrB(X))$
 and hence is bounded. 
 The local Lipschitz continuity now follows from the
 identity 
 $$ 
 (I - C_{J,x})^{-1} - (I- C_{J',x})^{-1} = 
 (I - C_{J,x})^{-1} 
 	\left(C_{J',x} - C_{J,x}\right) 
 (I- C_{J',x})^{-1} 
 $$
 (the ``second resolvent formula'')
 owing to the uniform bounds.

\end{proof}

%
%

\end{document}